\RequirePackage[l2tabu, orthodox]{nag} 
\documentclass{amsart}

\pdfoutput=1

\usepackage{amsmath, amssymb, comment}
\usepackage[pdftex,linktocpage,pdfstartview=FitH,colorlinks=true,allcolors=blue]{hyperref} 
\usepackage{amsthm}
\usepackage[utf8]{inputenc}
\usepackage[stretch=10, shrink=10]{microtype} 
\usepackage{todonotes} 
\usepackage{commath}
\usepackage{doi}
\usepackage[initials,msc-links,nobysame]{amsrefs}[2007/10/22]

\newcommand{\TitleWithUrl}[1]{\IfEmptyBibField{doi}%
  {\IfEmptyBibField{url}{\textit{#1}}%
    {\IfEmptyBibField{eprint}{\href {\BibField{url}}{\textit{#1}}}{\textit{#1}}}%
    }%
  {\href {https://doi.org/\BibField{doi}}{\textit{#1}}}}
\renewcommand{\eprint}[1]{\IfEmptyBibField{url}{\url{#1}}%
  {\href {\BibField{url}}{#1}}}

\BibSpec{article}{%
    +{}  {\PrintAuthors}                {author}
    +{,} { \TitleWithUrl}               {title}
    +{.} { }                            {part}
    +{:} { \textit}                     {subtitle}
    +{,} { \PrintContributions}         {contribution}
    +{.} { \PrintPartials}              {partial}
    +{,} { }                            {journal}
    +{}  { \textbf}                     {volume}
    +{}  { \PrintDate}                {date}
    +{,} { \issuetext}                  {number}
    +{,} { \eprintpages}                {pages}
    +{,} { }                            {status}
    +{,} { available at \eprint}        {eprint}
    +{}  { \parenthesize}               {language}
    +{}  { \PrintTranslation}           {translation}
    +{;} { \PrintReprint}               {reprint}
    +{.} { }                            {note}
    +{.} {}                             {transition}
    +{}  {\SentenceSpace \PrintReviews} {review}
}

\BibSpec{collection.article}{%
    +{}  {\PrintAuthors}                {author}
    +{,} { \TitleWithUrl}                     {title}
    +{.} { }                            {part}
    +{:} { \textit}                     {subtitle}
    +{,} { \PrintContributions}         {contribution}
    +{,} { \PrintConference}            {conference}
    +{}  {\PrintBook}                   {book}
    +{,} { }                            {booktitle}
    +{,} { \PrintDateB}                 {date}
    +{,} { pp.~}                        {pages}
    +{,} { }                            {status}
    +{,} { available at \eprint}        {eprint}
    +{}  { \parenthesize}               {language}
    +{}  { \PrintTranslation}           {translation}
    +{;} { \PrintReprint}               {reprint}
    +{.} { }                            {note}
    +{.} {}                             {transition}
    +{}  {\SentenceSpace \PrintReviews} {review}
}
\BibSpec{book}{%
    +{}  {\PrintPrimary}                {transition}
    +{,} { \TitleWithUrl}               {title}
    +{.} { }                            {part}
    +{:} { \textit}                     {subtitle}
    +{,} { \PrintEdition}               {edition}
    +{}  { \PrintEditorsB}              {editor}
    +{,} { \PrintTranslatorsC}          {translator}
    +{,} { \PrintContributions}         {contribution}
    +{,} { }                            {series}
    +{,} { \voltext}                    {volume}
    +{,} { }                            {publisher}
    +{,} { }                            {organization}
    +{,} { }                            {address}
    +{,} { \PrintDateB}                 {date}
    +{,} { }                            {status}
    +{}  { \parenthesize}               {language}
    +{}  { \PrintTranslation}           {translation}
    +{;} { \PrintReprint}               {reprint}
    +{.} { }                            {note}
    +{.} {}                             {transition}
    +{}  {\SentenceSpace \PrintReviews} {review}
}

\BibSpec{thesis}{%
    +{}  {\PrintAuthors}                {author}
    +{.} { \TitleWithUrl}               {title}
    +{:} { \textit}                     {subtitle}
    +{,} { \PrintThesisType}            {type}
    +{,} { }                            {organization}
    +{} { \PrintDate}                  {date}
    +{,} { }                            {address}
    +{,} { \eprint}                     {eprint}
    +{,} { }                            {status}
    +{}  { \parenthesize}               {language}
    +{}  { \PrintTranslation}           {translation}
    +{;} { \PrintReprint}               {reprint}
    +{.} { }                            {note}
    +{.} {}                             {transition}
    +{}  {\SentenceSpace \PrintReviews} {review}
}

\newcommand{\Ortho}{\operatorname{O}} 
\newcommand{\Hom}{\operatorname{Hom}} 
\newcommand{\quaddiff}{\mathcal{Q}} 
\newcommand{\confstr}{\mathcal{B}} 
\newcommand{\tr}{\textrm{tr}}

\numberwithin{equation}{section}

\theoremstyle{plain}
	\newtheorem{theorem}{Theorem}[section]
	\newtheorem{lemma}[theorem]{Lemma}
	\newtheorem{proposition}[theorem]{Proposition}
	\newtheorem{corollary}[theorem]{Corollary}

\theoremstyle{definition}
	\newtheorem{definition}[theorem]{Definition}

\theoremstyle{remark}
	\newtheorem{remark}[theorem]{Remark}
	\newtheorem{example}[theorem]{Example}

\title[Surfaces with spherical curvature lines]{Constrained elastic curves and surfaces with spherical curvature lines}

\author{Joseph Cho}
\address[J.~Cho]{Vienna University of Technology,
Wiedner Hauptstra\ss e 8-10/104, A-1040 Vienna. Austria}
\email{jcho@geometrie.tuwien.ac.at}

\author{Mason Pember}
\address[M.~Pember]{Dipartimento di Matematica, Politecnico di Torino,
Corso Duca degli Abruzzi 24, I-10129 Torino, Italy}
\email{mason.pember@polito.it}

\author{Gudrun Szewieczek}
\address[G.~Szewieczek]{Vienna University of Technology,
Wiedner Hauptstra\ss e 8-10/104, A-1040 Vienna. Austria}
\email{gudrun@geometrie.tuwien.ac.at}

\makeatletter
\@namedef{subjclassname@2020}{\textup{2020} Mathematics
  Subject Classification}
\makeatother

\begin{document}

\begin{abstract}
In this paper we consider surfaces with one or two families of spherical curvature lines. We show that every surface with a family 
of spherical curvature lines can locally be generated by a pair of initial data: a suitable curve of Lie sphere transformations and a 
spherical Legendre curve. We then provide conditions on the initial data for which such a surface is Lie applicable, an integrable class of surfaces that includes cmc and pseudospherical surfaces.
In particular we show that a Lie applicable surface with exactly one family of spherical curvature lines must be generated by the lift of a constrained 
elastic curve in some space form. In view of this goal, we give a Lie sphere geometric characterisation of constrained elastic curves via 
polynomial conserved quantities of a certain family of connections.
\end{abstract}

\maketitle


\section{Introduction}
\subsection{Background}
In 1691 Jacob Bernoulli posed the problem of finding the shape of an elastic rod; Daniel Bernoulli proposed that the shape taken can be modelled by curves $\gamma$, called \emph{elastic curves}, that are critical for the bending energy (also known as the total squared curvature) 
\begin{equation}
\label{eqn:bendingenergy}
\int_{\gamma} k^{2}ds,
\end{equation}
with respect to variations that fix length, where $s$ is the arc length parameter of $\gamma$, and Euler~\cite{E1744} gave a complete classification of these curves in the plane. Geometrically, such curves are characterised by the property that their curvature is proportional to the distance of the curve from a fixed line. These curves have since been generalised to Riemannian 2-dimensional space forms in~\cite{BG1986,LS1984i} and pseudo-Riemannian space forms (see for example~\cite{FGJL2006,J2004,P2018}). Moreover, these curves appear in various geometric contexts, such as in the construction of Willmore tori using the Hopf fibration in~\cite{P1985ii} and channel Willmore surfaces in~\cite{BG1986, H2003, LS1984ii, MN1999ii}. 

By considering the critical curves of~\eqref{eqn:bendingenergy} with respect to variations that fix length and enclosed area, we obtain a larger class of curves called \textit{constrained elastic curves}. These were studied in~\cite{BPP2008} and are used in~\cite{H2014} to create constrained Willmore tori via the Hopf fibration. The Euler-Lagrange equation for this variational problem is the stationary modified Korteweg-de-Vries (mKdV) equation. Thus constrained elastic curves are the stationary curves of the mKdV flow. In this paper we shall see that these curves appear in the study of Lie applicable surfaces with spherical curvature lines. 

On the other hand, our examination of surfaces with spherical curvature lines is rooted in the works of classical geometers in the 19th century and early 20th century.
Most of the work in the era focused on classifying surfaces with certain curvature restrictions which admit one family of planar or spherical curvature lines.
Bonnet considered minimal surfaces with planar curvature lines in \cite{B1855} (see also \cite{CO2017}).
It was Enneper who first looked into constant curvature surfaces with one family of spherical curvature lines in his work \cite{E1868}, pioneering the way for examining these surfaces via analytical methods.
Dobriner then used these methods to classify all pseudospherical surfaces and minimal surfaces with one family of spherical curvature lines in \cite{D1887i, D1887ii}.
The early developments of these analytical methods are well documented in the works of Enneper \cite{E1878}, Darboux \cite{D1889, D1896}, or Eisenhart \cite{E1909}, and these works had a profound impact in the modern surface theory scene.

A prototypical example of such modern development can be found in the study of constant mean curvature (cmc) surfaces in the Euclidean space.
A long standing problem of Hopf was to determine whether any cmc surface that is both compact and complete must necessarily be a sphere, and it was Wente in 1984 who was influenced by the aforementioned classical works to find a counterexample to this problem now known as the Wente torus in \cite{W1986}.
Abresch and Walter then independently noticed that the simplifying ansatz to facilitate finding such examples was precisely the geometric problem Enneper looked into: assuming that one family of curvature lines is planar.
Abresch in particular rediscovered the analytical methods of Enneper and applied them to finding cmc tori in \cite{A1987}; Walter noticed that the ansatz of one family of curvature lines being planar forces the other family of curvature lines to be spherical in \cite{W1987}, leading to an explicit parametrization of the cmc tori.
We note here that the classical geometers had already found these surfaces, but their compactness was not known until Wente's rediscovery.

The ansatz of Enneper, requiring that one family of curvature lines be spherical, has had pervasive modern influence, extending beyond the class of cmc surfaces in the Euclidean space.
Minimal tori in the $3$-sphere were discovered via the spherical curvature lines condition in \cite{Y1987}.
Isothermic tori were found by assuming that both families of curvature lines are spherical in \cite{B2001}.
Examples of linear Weingarten surfaces with spherical curvature lines were constructed in \cite{T2002, CFT2003} via Ribaucour transformations of the circular cylinder.
Surfaces with planar curvature lines were studied in the realm of Laguerre geometry in \cite{MN1999i}, and their applications to architecture have been explored in \cite{MDBL2018}.
Monge surfaces have one family of planar curvature lines, and their application was explored in \cite{BG2018}. Surfaces with two families of spherical curvature lines were shown to admit a $3$-parameter family of Lie deformations in~\cite{F2002}, and it is demonstrated how these surfaces relate to commuting Schr\"{o}dinger operators with magnetic fields. 

Many of the surfaces examined in the aforementioned works lie in the intersection of two larger classes of surfaces: surfaces with spherical curvature lines and Lie applicable surfaces \cite{F2002, MN2006}.
The class of Lie applicable surfaces is comprised of exactly the $\Omega$- and $\Omega_0$-surfaces, first discovered by Demoulin in \cite{D1911i, D1911ii, D1911iii}, and includes a wide range of well-studied surface classes: isothermic surfaces such as minimal and cmc surfaces, Guichard surfaces such as pseudospherical and non-cmc linear Weingarten surfaces, and $L$-isothermic surfaces such as minimal surfaces and linear Weingarten surfaces of Bryant type.
Thus, the examination of Lie applicable surfaces often leads to a uniform approach to considering these various surface classes, with their transformation theory being one of such examples \cite{BHPR2019}.

Lie applicable surfaces are Lie sphere invariant, allowing Lie sphere geometry to step in as the natural setting for examining these surfaces.
The comprehensive work of Blaschke \cite{B1929} paved the way for using the sphere geometries of M\"{o}bius, Laguerre, and Lie to study various surface classes. Blaschke derived invariants of surfaces in Lie sphere geometry that generically determine a surface up to Lie sphere transformations. The surfaces that are not determined are precisely the Lie applicable surfaces. Moreover, he showed how surfaces with a family of spherical curvature lines can be studied in this geometry via the use of sphere complexes, in particular showing that curves in this family are Lie sphere transformations of each other. 

The tools of Lie sphere geometry have also received multifarious modern interest, with a comprehensive introduction provided in \cite{C2008}. We now list some of such results that are particularly relevant to our work. Ribaucour transformations were given a Lie sphere geometric characterisation in \cite{BH2006}. Flat fronts in hyperbolic geometry and linear Weingarten surfaces in space forms were characterised via Lie sphere geometry in \cite{BHR2010, BHR2012}. Channel surfaces and their transformations were given a modern treatment in the realm of Lie sphere geometry in \cite{PS2018}. Lie applicable surfaces were shown to constitute an integrable system in~\cite{F2000ii} and their characterisation as the deformable surfaces of Lie sphere geometry was shown in~\cite{MN2006}. A gauge-theoretic characterisation of Lie applicable surfaces and their transformations was given in \cite{C2012i,P2020}, and the identification of subclasses of Lie applicable surfaces via polynomial conserved quantities was examined in \cite{BHPR2019}.

\subsection{Objectives}
Having reached an understanding of the historical and modern significance of surfaces with spherical curvature lines and the class of Lie applicable surfaces, we now state the aim of the paper.
We seek to provide a modern reinterpretation of Blaschke's methods of examining surfaces with spherical curvature lines, and use these methods to characterise all Lie applicable surfaces with spherical curvature lines.
In the pursuit of such a characterisation, we will see that constrained elastic curves play a central role.
Therefore, using Lie sphere geometry, we will:
\begin{itemize}
        \item characterise constrained elastic curves in Lie sphere geometry via polynomial conserved quantities of a certain $1$-parameter family of connections, 
	\item show that every surface with spherical curvature lines is generated by a certain curve of Lie sphere transformations applied to an initial spherical curve, and
	\item characterise Lie applicable surfaces with spherical curvature lines by restricting the curve of Lie sphere transformations and initial spherical curve, namely to a constrained elastic curve in an appropriate space form. 
\end{itemize}
Along the way we shall obtain additional results regarding surfaces with spherical curvature lines, for example showing that any such surface is a Ribaucour transform of a $3$-parameter family of channel surfaces. 

\subsection{Structure of the paper}
After a preparatory Section \ref{section2}, where we introduce the basics of Lie sphere geometry and the gauge-theoretic description of Lie applicable surfaces, we restrict our attention to the theory of curves in $2$-dimensional Riemannian and Lorentzian space forms in Section \ref{section3}. In Theorem~\ref{thm:elasticcomplex} we prove that constrained elastic curves can be characterised by a certain enveloped circle congruence belonging to a circle complex. A gauge-theoretic approach to curve theory was given in the realm of M\"{o}bius geometry in \cite{BHMR2016}, which we apply to our situation of Lie sphere geometry, culminating in the characterisation of constrained elastic curves in terms of polynomial conserved quantities in Theorem \ref{thm:constrained}, analogous to the approaches of~\cite{BS2012, BHPR2019}.

We then switch our attention to surfaces with spherical curvature lines in Section~\ref{section4}, where we mainly focus on reinterpreting Blaschke's work in modern context.
After characterising surfaces with spherical curvature lines in Lie sphere geometry in Theorem \ref{thm:Lsph}, we show that every surface with spherical curvature lines is obtained via a certain $1$-parameter family of Lie sphere transformations of a given spherical curve in Theorem \ref{thm:A}, which we call \emph{spherical evolution map} (see Definition \ref{def:sem}).
We also recover the \emph{osculating sphere complexes} of Blaschke, obtaining another Lie sphere geometric characterisation of surfaces with spherical curvature lines in Proposition \ref{prop:sphLi}.
Finally, we break symmetry to conformal geometry, Laguerre geometry, and Euclidean geometry to characterise surfaces with orthogonally intersecting spherical curvature lines, surfaces with planar curvature lines, and Monge surfaces, respectively, in Section \ref{subsec:sphsymbreak}. This then enables us to recover Blaschke's classification of surfaces with two families of spherical curvature lines. 

In Section \ref{section5}, we examine the relationship between surfaces with spherical curvature lines and Ribaucour transformations.
Blaschke showed that Ribaucour transformations of a channel surface must have one family of spherical curvature lines.
We consider the converse statement, and show that every surface with spherical curvature lines must be a Ribaucour transform of a $3$-parameter family of channel surfaces in Theorem \ref{thm:3paramchannel}.
We then provide an alternative geometric proof of this result using the spherical evolution map.

Finally, in Section \ref{sec:lieappsph}, we consider Lie applicable surfaces with spherical curvature lines. We obtain a characterisation of such surfaces in Theorem~\ref{thm:lieappsph}, showing that either both families of curvature lines are spherical or that the surface is obtained by a certain evolution of a constrained elastic curve. 
We then further examine the case that both families of curvature lines are spherical, using the curved flat formalism of~\cite{BP2020}. In so doing we show that Joachimsthal surfaces are characterised as the Darboux transforms of Dupin cyclides in Theorem~\ref{thm:joachimsthal}. Finally, we show that Calapso transforms of Lie applicable surfaces preserve the property of spherical curvature lines.

\subsection{Acknowledgements.} 
We would like to thank Shoichi Fujimori, Udo Hertrich-Jeromin, Emilio Musso, Yuta Ogata and Álvaro Pámpano for pleasant and fruitful conversations about the topics of this paper. 

We gratefully acknowledge the financial support of the following research grants: the JSPS/FWF Joint Project I3809-N32 ``Geometric shape generation''; the JSPS Grant-in-Aid for JSPS Fellows 19J10679; the FWF research project P28427-N35 ``Non-rigidity and symmetry breaking''; GNSAGA of INdAM and the MIUR grant ``Dipartimenti di Eccellenza'' 2018 - 2022, CUP: E11G18000350001, DISMA, Politecnico di Torino.

\section{Lie sphere geometry}\label{section2}

Let $n\in \mathbb{N}$ and let $\mathbb{R}^{n+2,2}$ denote the $(n+4)$-dimensional pseudo-Euclidean space equipped with a non-degenerate symmetric bilinear form $( , )$ of signature $(n+2,2)$. Denoting by $\mathcal{L}$ the light cone of $\mathbb{R}^{n+2,2}$, the Lie quadric $\mathbb{P}(\mathcal{L})$ parametrises oriented hyperspheres of $(n+1)$-dimensional space forms (see for example~\cite{C2008}). Two points in the Lie quadric lie on a line in $\mathbb{P}(\mathcal{L})$ if and only if the corresponding spheres are in oriented contact. Thereby, the Grassmannian of isotropic $2$-planes in $\mathbb{R}^{n+2,2}$, denoted by $\mathcal{Z}$, represents the manifold of contact elements.

In this setting, the orthogonal group $\Ortho(n+2,2)$ acts transitively on $\mathcal{L}$, and is a double cover for the group of Lie sphere transformations.
Throughout the paper, we identify the Lie algebra $\mathfrak{o}(n+2,2)$ with the exterior algebra $\wedge^2 \mathbb{R}^{n+2,2}$ via
	\[
		(a \wedge b)c = (a, c) b - (b, c) a
	\]
for $a,b,c \in \mathbb{R}^{n+2,2}$.

Let $\Sigma$ be a simply-connected $n$-dimensional manifold. A smooth map $f:\Sigma\to \mathcal{Z}$ can be considered as a rank 2 null subbundle of the trivial bundle $\underline{\mathbb{R}}^{n+2,2} := \Sigma\times \mathbb{R}^{n+2,2}$. We shall denote by $f^{(1)}$ the set of sections of $f$ and derivatives of sections of $f$. In Lie sphere geometry, one studies hypersurfaces by considering their contact lifts which in this setting amounts to the following definition:  

\begin{definition}
A map $f : \Sigma \to \mathcal{Z}$ is called a \emph{Legendre map} if $f^{(1)} \le f^\perp$ and, for any $x\in \Sigma$, if $X\in T_{x}\Sigma$ such that $d\sigma(X) \in \Gamma f$ for all $\sigma\in \Gamma f$ then $X = 0$.
\end{definition}

For any rank 2 null subbundle $f$ of $\underline{\mathbb{R}}^{n+2,2}$ we have that $f^{\perp}/f$ is a rank $n$ subbundle of $\underline{\mathbb{R}}^{n+2,2}/f$ that inherits a positive definite metric from $\mathbb{R}^{n+2,2}$. When $f$ is a Legendre map, we have that $f^{(1)} = f^{\perp}$. 

Suppose that $f : \Sigma \to \mathcal{Z}$ is a Legendre map. Fix a point $x\in \Sigma$. We say that a 1-dimensional subspace $s(x)\in f(x)$ is a \textit{curvature sphere of $f$ at $x$} if there exists a non-zero subspace $T_{s(x)}\le T_{x}\Sigma$ such that $d_{X}\sigma \in f(x)$ for any $X\in T_{s(x)}$ and $\sigma\in \Gamma f$ such that $\sigma(x)\in s(x)$. We call the maximal $T_{s(x)}$ the \textit{curvature space of $s(x)$}. 
We say that a point $x\in \Sigma$ is \textit{umbilic} if there exists exactly one curvature sphere $s(x)\le f(x)$, in which case $T_{s(x)} = T_{x}\Sigma$.

\subsection{Symmetry breaking}
\label{subsec:symbreak}
As seen in~\cite{BHPR2019,C2008} one can break the symmetry of Lie sphere geometry to obtain conformal geometry and space form geometries. We briefly recall this process in this subsection.

Fix $\mathfrak{p} \in \mathbb{R}^{n+2,2}$ with $\chi:= - (\mathfrak{p}, \mathfrak{p})= \pm 1$. If $\chi = 1$ then $\langle\mathfrak{p}\rangle^{\perp}\cong \mathbb{R}^{n+2,1}$ defines a Riemannian conformal geometry (see for example~\cite{H2003}), whereas if $\chi = -1$ then $\langle\mathfrak{p}\rangle^{\perp}\cong \mathbb{R}^{n+1,2}$ defines a Lorentzian conformal geometry. We call $\mathfrak{p}$ the \textit{point sphere complex} of the conformal geometry defined by $\langle\mathfrak{p}\rangle^{\perp}$. Lie sphere transformations $A\in \Ortho(n+2,2)$ satisfying $A\mathfrak{p} = \mathfrak{p}$ are then identified with conformal transformations. 

In the case that $\chi = 1$, the conformal $(n+1)$-sphere $S^{n+1}$ is identified with $\mathbb{P}(\mathcal{L}) \cap \langle\mathfrak{p}\rangle^\perp$. Let $\pi_{\mathfrak{p}}: \mathbb{R}^{n+2,2}\to \langle\mathfrak{p}\rangle^{\perp}$ be the projection onto $\langle\mathfrak{p}\rangle^{\perp}$. A point $s \in \mathbb{P}(\mathcal{L})\backslash \mathbb{P}(\mathcal{L}) \cap \langle\mathfrak{p}\rangle^\perp$ projects to a spacelike $1$-dimensional subbundle $S = \pi_{\mathfrak{p}} (s)$, and thus such points $s$ represent hyperspheres of this conformal geometry. Note that $\pi_{\mathfrak{p}}^{-1}(S) = s\oplus \langle\mathfrak{p}\rangle$ is a $2$-dimensional subspace of $\mathbb{R}^{n+2,2}$ with signature $(1,1)$. Thus there are two points in $\mathbb{P}(\mathcal{L})\backslash \mathbb{P}(\mathcal{L} )\cap \langle\mathfrak{p}\rangle^\perp$ that project to $S$, giving rise to a notion of orientation for hyperspheres (see~\cite{C2008}). 

We may break symmetry further by choosing a non-zero vector $\mathfrak{q}\in \langle\mathfrak{p}\rangle^{\perp}$ and defining 
\[ \mathfrak{Q}^{n+1} := \{y\in \mathcal{L}: (y,\mathfrak{q})=-1, (y,\mathfrak{p})=0\},\]
which is isometric to a space form with constant sectional curvature $\kappa = - (\mathfrak{q},\mathfrak{q})$. If $\chi =1$, then $\mathfrak{Q}^{n+1}$ is a Riemannian space form and if $\chi=-1$, then $\mathfrak{Q}^{n+1}$ is a Lorentzian space form. We call $\mathfrak{q}$ the \textit{space form vector} of $\mathfrak{Q}^{n+1}$. The quadric 
\[ \mathfrak{P}^{n+1}:= \{y\in \mathcal{L}: (y,\mathfrak{p})=-1, (y,\mathfrak{q})=0\}\]
is then identified with the space of hyperplanes (complete, totally geodesic hypersurfaces) in this space form. The Lie sphere transformations $A\in \Ortho(n+2,2)$ satisfying $A\mathfrak{p} = \mathfrak{p}$ and $A\mathfrak{q} = \mathfrak{q}$ are identified with the isometries of $\mathfrak{Q}^{n+1}$. 

Given a Legendre map $f:\Sigma\to \mathcal{Z}$ we call the projection to a space form $\mathfrak{f}:=f\cap\mathfrak{Q}^{n+1}$ the \textit{point sphere map} and $\mathfrak{t}:= f\cap \mathfrak{P}^{n+1}$ the \textit{tangent plane congruence} of $f$ with respect to $\mathfrak{p},\mathfrak{q}$. 

\begin{remark}
In~\cite{BHPR2019} it is shown how a non-zero lightlike vector $\mathfrak{q}\in \mathcal{L}$ defines a Laguerre subgeometry. One then has that the Lie sphere transformations $A\in \Ortho(n+2,2)$ for which $\mathfrak{q}$ is an eigenvector represent Laguerre transformations. 
\end{remark}

\subsection{Linear sphere complexes}
\label{subsec:linsphcom}

A central concept to this paper is the notion of linear sphere complexes developed in \cite{L1872} (see also \cite{B1929,H2003, RS2020}).

\begin{definition}
A 1-dimensional subspace $L\le \mathbb{R}^{n+2,2}$ defines a \textit{linear sphere complex} $E_{L}:= \mathbb{P}(\mathcal{L})\cap L^{\perp}$, a $(n+1)$-dimensional family of hyperspheres. If $L$ is spacelike then we say that $E_{L}$ is an \emph{elliptic} linear sphere complex. 
\end{definition}

Suppose that $L$ defines an elliptic linear sphere complex. After a choice of a timelike point sphere complex $\mathfrak{p}$, one has that $L\oplus\langle \mathfrak{p}\rangle$ is a $(1,1)$ plane. Thus there exist two hyperspheres $L^{\pm}\in \mathbb{P}(\mathcal{L})$ such that $L^{+}\oplus L^{-}= L\oplus\langle \mathfrak{p}\rangle$. These represent the same hypersphere $S = \pi_{\mathfrak{p}}(L)$ in the conformal geometry $\langle\mathfrak{p}\rangle^{\perp}$ with opposite orientations. The hyperspheres belonging to $E_{L}$ are then the hyperspheres intersecting $S$ transversally at a fixed non-zero angle. 

\begin{remark}
\label{rem:spcomplex}
In the case that $L\perp \mathfrak{p}$, one has that the hyperspheres belonging to $E_{L}$ intersect $S$ orthogonally. In the presence of a space form vector $\mathfrak{q}$, $L\perp \mathfrak{q}$ implies that the hypersphere defined by $S$ represents a hyperplane in $\mathfrak{Q}^{n+1}$. 
\end{remark}

\subsection{Lie applicable Legendre maps}

Demoulin~\cite{D1911i,D1911ii,D1911iii} discovered a class of surfaces belonging to Lie sphere geometry that constitute an integrable system. In~\cite{C2012i,P2020} these hypersurfaces were given a gauge theoretic interpretation and we shall recall that here. 

\begin{definition}
\label{def:lieapp}
A Legendre map $f:\Sigma\to\mathcal{Z}$ is \textit{Lie applicable} if there exists $\eta\in \Omega^{1}(f\wedge f^{\perp})$ satisfying $d\eta = [\eta\wedge \eta] = 0$ such that the quadratic differential $\quaddiff$ defined by 
\[ \quaddiff^{\eta}(X,Y) = \tr(f\to f: \sigma \mapsto \eta(X)d_{Y}\sigma)\]
is non-zero. We call such an $\eta$ a \textit{gauge potential} of $f$. 
\end{definition}

Gauge potentials are not unique: given a gauge potential $\eta$ one has that $\tilde{\eta}:= \eta - d\tau$ is a gauge potential for any $\tau \in \Gamma (\wedge^{2}f)$. We say that $\tilde{\eta}$ is \textit{gauge equivalent} to $\eta$ and we denote by $[\eta]$ the equivalence class of gauge potentials that are gauge equivalent to $\eta$, called the \textit{gauge orbit} of $\eta$. The quadratic differential $\quaddiff^{\eta}$ is invariant of gauge transformation, i.e., $\quaddiff^{\tilde{\eta}} = \quaddiff^{\eta}$. 

Given a Lie applicable surface $f$ with gauge potential $\eta$, we have that $\{d+t\eta\}_{t\in\mathbb{R}}$ is a 1-parameter family of flat metric connections. It then follows that there exist local orthogonal trivialising gauge transformations $T(t):\Sigma\to \Ortho(n+2,2)$, that is, 
\[ T(t)\cdot (d+t\eta) = d.\]
These gauge transformations give new Lie applicable Legendre maps via $f^{t}:= T(t) f$, which are called \textit{Calapso transforms of $f$}. 

Since $d+m\eta$ is a flat metric connection for any $m\in\mathbb{R}\backslash\{0\}$, there exist many parallel sections. Suppose that $\hat{s}$ is a parallel rank 1 null subbundle of $d+m\eta$ such that $\hat{s}$ is nowhere orthogonal to $f$. Let $\hat{f}:= \hat{s}\oplus s_{0}:\Sigma\to \mathcal{Z}$, where $s_{0}:= f\cap \hat{s}^{\perp}$. These $\hat{f}$ are Lie applicable Legendre maps, called \textit{$m$-Darboux transforms of $f$}. One also has that $f$ is an $m$-Darboux transform of $\hat{f}$. Thus, we refer to $(f,\hat{f})$ as an \textit{$m$-Darboux pair}.

\section{Curves in Lie sphere geometry}\label{section3}

In this section we shall consider the case that $n=1$. Thus $\mathbb{P}(\mathcal{L})$ parametrises circles in $2$-dimensional space forms such as the Euclidean $2$-plane. This was utilised by Blaschke~\cite[Chapter 5]{B1929} to study curves. See~\cite{BLPT2020} for a recent application of planar Lie sphere geometry.

Assume that $I:= \Sigma$ is an open interval. We then have that Legendre maps $C:I\to \mathcal{Z}$ project to fronts, that is, spacelike curves with admissible singularities in Riemannian or Lorentzian space forms (see~\cite{AGV1985}). We thus refer to Legendre maps in this setting as \textit{Legendre curves}. In this case the ``hyperspheres'' enveloped by $C$ are circles. At each point $x\in I$, there exists exactly one curvature sphere $c(x)$ which corresponds to the osculating circle. This gives rise to a rank 1 subbundle $c\le C$ satisfying $d\sigma\in \Omega^{1}(C)$ for any section $\sigma\in \Gamma c$. If at some point $x\in I$ we have that $d\sigma(T_{x}I)\le c(x)$, for any $\sigma \in \Gamma c$, then we say that $x$ is an \textit{inflection point} of $C$. If every point of $I$ is an inflection point then $c$ is constant and we say that $C$ is \textit{circular}. 

\begin{lemma}
\label{lem:circ}
$C\perp \mathfrak{q}$ for some non-zero vector $\mathfrak{q}\in \mathbb{R}^{3,2}$ if and only if $C$ is circular and $\mathfrak{q}\in \Gamma C$. 
\end{lemma}
\begin{proof}
If $C$ is circular then $c\le C$ is constant and we may choose a constant $\mathfrak{q}\in \Gamma c$. 

Conversely, if $C\perp \mathfrak{q}$ for some constant $\mathfrak{q}$, then it follows that $C^{(1)}\perp \mathfrak{q}$. Since $C$ is Legendre one has that $C^{(1)} = C^{\perp}$. On the other hand $(C^{\perp})^{\perp}=C$. Thus $\mathfrak{q}\in \Gamma C$. 
\end{proof}

\begin{remark}
Notice that by Lemma~\ref{lem:circ} we have that $C\perp \mathfrak{q}$ necessarily implies that $\mathfrak{q}$ is lightlike. If $\mathfrak{q}$ is timelike, then $C(x)\not\perp\mathfrak{q}$ for any $x\in I$, whereas if $\mathfrak{q}$ is spacelike, then the set of points $x\in I$ where $C(x)\perp \mathfrak{q}$ is discrete.  
\end{remark}

\begin{lemma}
\label{lem:W0s}
Let $W_{0}$ be a constant $3$-dimensional subspace of $\mathbb{R}^{3,2}$ and suppose that $c_{0}:=C\cap W_{0}$ is a rank 1 subbundle of $C$. Then $C$ is circular. 
\end{lemma}
\begin{proof}
If $C^{\perp}\cap W_{0}$ is rank $1$ then, since $C^{(1)}=C^{\perp}$, it follows that $c_{0}$ is constant and thus $C$ is circular. If $C^{\perp} = W_{0}$ then $C$ fails to be a Legendre curve. If $C^{\perp}\cap W_{0}$ is rank $2$ then $\tilde{c}_{0}:= C\cap W_{0}^{\perp}$ has rank $1$ and it follows that $\tilde{c}_{0}$ is constant. Hence, $C$ is circular.   
\end{proof}

\subsection{Polarisations and polynomial conserved quantities}

Since $TI$ is rank $1$, any $\xi\in \Omega^{1}(C\wedge C^{\perp})$ satisfies $d\xi = [\xi\wedge \xi]=0$. Therefore any $\xi\in \Omega^{1}(C\wedge C^{\perp})$ with non-zero quadratic differential\footnote{Note that, unlike the $n=2$ case, $\quaddiff^{\xi}$ does not determine $[\xi]$.} $\quaddiff^{\xi}$ is a gauge potential. In analogy to~\cite{BHMR2016} we refer to gauge orbits $[\xi]$ in this setting as \textit{polarisations of $C$}. 

\begin{example}
Let $\mathfrak{p}, \mathfrak{q}$ be a point sphere complex and space form vector for a space form $\mathfrak{Q}^{2}$ and let $\mathfrak{f} = C\cap \mathfrak{Q}^{2}$ denote the corresponding point sphere map. We then call the gauge orbit $[\xi]$ with $\xi := - \mathfrak{f}\wedge d\mathfrak{f}$ the \textit{arclength polarisation of $C$ with respect to $\mathfrak{p}, \mathfrak{q}$}. The corresponding quadratic differential is given by $\quaddiff^{\xi}= (d\mathfrak{f},d\mathfrak{f})$. 
\end{example}

In analog to~\cite{BS2012}, we have the following definition: 
\begin{definition}
We say that $p(t) = p_{0} + tp_{1} + ... + t^{d}p_{d}\in (\Gamma \underline{\mathbb{R}}^{3,2})[t]$ is called \textit{a polynomial conserved quantity of $d+t\xi$ of degree $d$} if $(d+t\xi)p(t) = 0$ for all $t\in\mathbb{R}$ and $p_{d}\in \Gamma C$.
\end{definition}

If $\tilde{\xi} = \xi - d\tau \in [\xi]$, then $d+t\tilde{\xi} = \exp(t\tau)\cdot (d+t\xi)$. Thus if $p(t)$ is a a polynomial conserved quantity of $d+t\xi$ of degree $d$, then $\tilde{p}(t) := \exp(t\tau)p(t)$ is a polynomial conserved quantity of $d+t\tilde{\xi}$ of degree $d$. Note that without the requirement that $p_{d}\in \Gamma C$, we would have that the degree of polynomial conserved quantities is not invariant under gauge transformation. 

Evaluation of the coefficients of the polynomial $(d+t\xi)p(t)=0$ yields the following lemma, analogous to~\cite[Proposition 2.2]{BS2012}: 

\begin{lemma}
$p(t) = p_{0} + tp_{1} + ... + t^{d}p_{d}$ is a polynomial conserved quantity of $d+t\xi$ of degree $d$ if and only if \begin{itemize}
\item $p_{0}$ is constant, 
\item $dp_{i}+\xi p_{i-1}=0$ for all $i\in \{1,...,d\}$, and 
\item $p_{d}\in \Gamma C$. 
\end{itemize}
\end{lemma}

Clearly constant linear combinations of polynomial conserved quantities of degree $d$ are polynomial conserved quantities of degree $d$. Hence the space of polynomial conserved quantities of $d+t\xi$ of degree $d$ forms a vector space over $\mathbb{R}$. In this paper we will focus on spaces of linear conserved quantities. 

\begin{lemma}
\label{lem:lcqspan}
Suppose that $p$ and $\tilde{p}$ are linear conserved quantities of $C$ with $p(0) = \tilde{p}(0)$. Then either $p = \tilde{p}$ or $C$ is circular.
\end{lemma} 
\begin{proof}
We may write $p = \mathfrak{p} + t \sigma$ and $\tilde{p} = \mathfrak{p} + t\tilde{\sigma}$ for some constant $\mathfrak{p}\in \mathbb{R}^{3,2}$ and sections $\sigma, \tilde{\sigma}\in \Gamma C$. If $\sigma = \tilde{\sigma}$ then $p = \tilde{p}$. Otherwise we have that $p - \tilde{p} = t(\sigma - \tilde{\sigma})$ is a conserved quantity of $d+t\xi$. This then implies that $d(\sigma - \tilde{\sigma}) = 0$. Hence $\sigma - \tilde{\sigma}\in \Gamma C$ is constant, implying that $C$ is circular. 
\end{proof}

The next lemma shows that for any curve $C$, we can always find polarisations $[\xi]$ such that $d+t\xi$ admits $2$-dimensional spaces of linear conserved quantities. 

\begin{lemma}
\label{lem:2dimlcq}
Suppose that $W_{0}\le \mathbb{R}^{4,2}$ is a $2$-dimensional subspace with $W_{0}\cap C^{\perp}=\{0\}$. Then there exists a polarisation $[\xi]$ of $C$ such that the space $W$ of linear conserved quantities of $d+t\xi$ satisfies $W_{0}\le W(0)$. 
\end{lemma}
\begin{proof}
Fix $\mathfrak{q},\tilde{\mathfrak{q}}\in W_{0}$. Then, by the assumption $W_{0}\cap C^{\perp}=\{0\}$, we can find sections $\sigma, \tilde{\sigma}\in \Gamma C$ such that 
\[ (\sigma,\mathfrak{q}) = (\tilde{\sigma},\tilde{\mathfrak{q}}) = 0 \quad \text{and}\quad (\tilde{\sigma},\mathfrak{q})= (\sigma,\tilde{\mathfrak{q}}) = -1 .\]
Now define $\xi :=- \sigma\wedge d\sigma - \tilde{\sigma}\wedge d\tilde{\sigma}$. Then $\quaddiff^{\xi} = (d\sigma,d\sigma) + (d\tilde{\sigma},d\tilde{\sigma}) \neq 0$, as otherwise $C$ would fail to be a Legendre curve. Hence $[\xi]$ is a polarisation of $C$. Moreover, 
\[ q(t):= \mathfrak{q} +t \tilde{\sigma} \quad \text{and}\quad \tilde{q}(t):= \tilde{\mathfrak{q}} + t\sigma\]
are linear conserved quantities of $d+t\xi$ satisfying $q(0) = \mathfrak{q}$ and $\tilde{q}(0)= \tilde{\mathfrak{q}}$. Hence $W:= \langle q, \tilde{q}\rangle$ satisfies $W_{0}\le W(0)$.  
\end{proof}

Let $W$ be a vector space of linear conserved quantities of $d+t\xi$. Since $d+t\xi$ is a flat metric connection, we have that the coefficients of the polynomial $(p(t),q(t))$ are constant for any $p,q\in W$. We may now equip $W$ with a pencil of metrics $\{(,)_{t}\}_{t\in \mathbb{R}\cup \{\infty\}}$ where 
\[ ( p,q)_{t} :=(p(t),q(t))\]
for $t\in \mathbb{R}$ and $(,)_{\infty} := \lim_{t\to \infty} \frac{1}{t}(,)_{t}$. Writing $p = p_{0} + tp_{1}$ and $q =q_{0}+tq_{1}$ we have that 
\[ (p,q)_{\infty} = (p_{0},q_{1}) + (q_{0},p_{1}).\] 
Note that this pencil of metrics is invariant under gauge transformation.

\begin{lemma}
\label{lem:infmet}
Suppose that $C$ is non-circular and that $W$ is a $2$-dimensional space of linear conserved quantities of $d+t\xi$. Then $(,)_{\infty}\neq 0$. 
\end{lemma}
\begin{proof}
Suppose that $(,)_{\infty}\equiv 0$ and $p,q\in W$ are linearly independent. Writing $p(t)= p_{0}+tp_{1}$ and $q(t) = q_{0}+tq_{1}$ we have that $p_{0}$ and $q_{0}$ are linearly independent by Lemma~\ref{lem:lcqspan} and, since $(,)_{\infty}\equiv 0$, we have that 
\begin{equation} 
\label{eqn:inf0}
0 = (p_{0},p_{1}) = (q_{0},q_{1}) = (p_{0},q_{1})+(q_{0},p_{1}).
\end{equation}
If $(p_{0},q_{1}) = 0$ then we can deduce that $C$ is circular from Lemma~\ref{lem:W0s}. Hence, we must have that $(p_{0},q_{1}) \neq 0$. Now using $p_{1},q_{1}$ as a basis for $C$ we can compute the quadratic differential of $\xi$ to be
\begin{align*}
 \quaddiff^{\xi}(X,X) &= \frac{1}{(p_{1},q_{0})}(\xi(X)d_{X}p_{1}, q_{0}) + \frac{1}{(q_{1},p_{0})}(\xi(X)d_{X}q_{1}, p_{0}) \\
 &= \frac{1}{(p_{1},q_{0})}( (\xi(X)d_{X}p_{1}, q_{0}) - (\xi(X)d_{X}q_{1}, p_{0})). 
 \end{align*}
 On the other hand, since $p,q$ are linear conserved quantities of $d+t\xi$, we have that 
 \[ dp_{1}+\xi p_{0}=dq_{1}+\xi q_{0} = 0,\]
 from which one deduces that $\quaddiff^{\xi}(X,X)=0$, contradicting that $\xi$ is a polarisation of $C$. 
\end{proof}

Let $\xi = - \mathfrak{f}\wedge d\mathfrak{f}$ be the arclength polarisation of a Legendre curve $C$ into a space form $\mathfrak{Q}^{2}$. Then $d+t\xi$ admits a two dimensional space of linear conserved quantities $W = \langle p, q\rangle$ where $p(t) = \mathfrak{p}$ and $q(t) = \mathfrak{q}- t\mathfrak{f}$. Moreover, $(p,p)_{0}\neq 0$, $(q,q)_{\infty}=2$ and $(p,p)_{\infty} = (p,q)_{\infty}=(p,q)_{0} = 0$. 

Conversely, suppose that $W = \langle p, q\rangle$ is a two dimensional space of linear conserved quantities of $d+t\xi$ satisfying $(p,p)_{0}\neq 0$, $(q,q)_{\infty}\neq 0$ and $(p,p)_{\infty} = (p,q)_{\infty}=(p,q)_{0} = 0$. By reparameterising $t$, we may assume that $(q,q)_{\infty} = 2$. Let $\mathfrak{Q}^{2}$ be the space form with point sphere complex $\mathfrak{p}:= p(0)$ and space form vector $\mathfrak{q}:= q(0)$ and let $\mathfrak{f}=C\cap \mathfrak{Q}^{2}$ and $\mathfrak{t} = C\cap \mathfrak{P}^{2}$ be the corresponding point sphere map and tangent plane congruence. Define $\tau = \mathfrak{t}\wedge p_{1}$. Then $\tilde{p}(t) := \exp(t\tau)p(t) = \mathfrak{p}$. With $\tilde{q}(t) :=  \exp(t\tau)q(t) = \mathfrak{q} + t\tilde{q}_{1}$, the conditions $(\tilde{q},\tilde{q})_{\infty} = 2$ and $(\tilde{p},\tilde{q})_{\infty}=0$ imply that $\tilde{q}_{1} = -\mathfrak{f}$. The conditions 
\[ d\tilde{p}_{1}+\xi \mathfrak{p} = d\tilde{q}_{1} + \xi \mathfrak{q} = 0\]
then imply that $\xi = -\mathfrak{f}\wedge d\mathfrak{f}$. 

We thus arrive at the following lemma: 
\begin{lemma}
\label{lem:arcpol}
Suppose that $W = \langle p,q\rangle$ is a $2$-dimensional space of linear conserved quantities of $d+t\xi$ such that $(p,p)_{0}\neq 0$, $(q,q)_{\infty}\neq 0$ and $(p,p)_{\infty} = (p,q)_{\infty}= (p,q)_{0}=0$. Then, after possibly rescaling $\xi$, we have that $[\xi]$ is the arclength polarisation of $C$ with respect to $\mathfrak{p}, \mathfrak{q}$ where $\mathfrak{p}:= p(0)$ and $\mathfrak{q}:= q(0)$. 
\end{lemma}

\begin{lemma}
\label{lem:arcpol2}
Suppose that $C$ is non-circular and that $W$ is a $3$-dimensional space of linear conserved quantities of $d+t\xi$. Then $[\xi]$ is the arclength polarisation of $C$ with respect to some $\mathfrak{p},\mathfrak{q}\in W(0)$. 
\end{lemma}
\begin{proof}
By Lemma~\ref{lem:lcqspan} we have that $W(0)$ is a $3$-dimensional subspace of $\mathbb{R}^{3,2}$. On the other hand $C^{\perp}$ is a $3$-dimensional subbundle of $\underline{\mathbb{R}}^{3,2}$. Thus with $\nu:= C^{\perp}\cap W(0)$ we have that $\nu(x)\neq \{0\}$ for all $x\in I$. By Lemma~\ref{lem:W0s} we must have for some $x\in I$ that $\nu(x)\not\le C(x)$. Thus there exists a vector $\mathfrak{p}\in W(0)$ with $\mathfrak{p}\in (C(x))^{\perp}$ and $(\mathfrak{p},\mathfrak{p})\neq 0$. We can then find $p\in W$ such that $p(0)=\mathfrak{p}$. Since the top term of $p$ lies in $C$ and $\mathfrak{p}\in (C(x))^{\perp}$ it follows that $(p,p)_{\infty} =0$. Denote by $p^{\perp_{\infty}}$ and $p^{\perp_{0}}$ the subspaces of $W$ that are perpendicular to $p$ with respect to $(,)_{\infty}$ and $(,)_{0}$. These are $2$-dimensional subspaces of $W$, thus there exists a non-zero $q\in p^{\perp_{\infty}}\cap p^{\perp_{0}}$. By Lemma~\ref{lem:infmet} we must have that $(q,q)_{\infty}\neq 0$, as otherwise $(,)_{\infty}$ would vanish on $\langle p,q\rangle$. Thus, by Lemma~\ref{lem:arcpol}, $[\xi]$ is the arclength polarisation of $C$ with respect to $\mathfrak{p}$ and $\mathfrak{q}:=q(0)$.
\end{proof}

\subsection{Constrained elastic curves}

Fix a point sphere complex $\mathfrak{p}$ and space form vector $\mathfrak{q}$ for a space form $\mathfrak{Q}^{2}$ with sectional curvature $\kappa = -(\mathfrak{q},\mathfrak{q})$ and $\chi = -(\mathfrak{p},\mathfrak{p})$. The corresponding point sphere map $\mathfrak{f}$ and tangent plane congruence $\mathfrak{t}$ of a Legendre curve $C:I\to\mathcal{Z}$ satisfy the relation $d\mathfrak{t} = - k d\mathfrak{f}$, where $k$ is the curvature of $\mathfrak{f}$. The curvature sphere of $f$ is thus given by $c = \langle \mathfrak{t}+k\mathfrak{f}\rangle$. Hence, $C$ is circular if and only if $k$ is constant. 

Let $'$ denote differentiation with respect to the arclength parameter of $\mathfrak{f}$. Then according to~\cite{BPP2008,H2014} (the Lorentzian case can be deduced from~\cite{P2018} and the references therein), $\mathfrak{f}$ is a \textit{constrained elastic curve} in $\mathfrak{Q}^{2}$ if and only if 
\begin{equation}
\label{eqn:elsecond} 
k'' + \chi\frac{k^{3}}{2}  + (\mu+\kappa)k + \lambda=0,
\end{equation}
for some $\mu, \lambda\in \mathbb{R}$. If $\lambda=0$ then $\mathfrak{f}$ is an \textit{elastic curve} and if $\lambda = \mu = 0$ then $\mathfrak{f}$ is a \textit{free elastic curve}. Clearly circular curves, that is when $k$ is constant, are constrained elastic curves. 

Define a vector 
\begin{equation}
\label{eqn:c}
\mathfrak{r} := -(k\kappa+\lambda)\mathfrak{f} + (\tfrac{k^{2}}{2} \chi + \mu)\mathfrak{t} - k' \mathfrak{f}' + k\mathfrak{q} - \tfrac{k^{2}}{2}\mathfrak{p},
\end{equation}
for some constants $\lambda, \mu\in \mathbb{R}$. Noting that 
\[ \mathfrak{f}'' = - \kappa \mathfrak{f} +\chi k\mathfrak{t} + \mathfrak{q} - k \mathfrak{p}, \]
we can calculate that
\[ \mathfrak{r}' = (-k'' - k\kappa -\lambda - \tfrac{k^3}{2}\chi - \mu k)\mathfrak{f}'.\]
It then follows that $\mathfrak{r}'=0$ if and only if $\mathfrak{f}$ is a constrained elastic curve satisfying~\eqref{eqn:elsecond}. Moreover, one has that $(\mathfrak{t} + \frac{k}{2}\mathfrak{f},\mathfrak{r})=0$. 

Conversely, suppose that $(\mathfrak{t} + \frac{k}{2}\mathfrak{f},\mathfrak{r})=0$, for some non-zero constant $\mathfrak{r}\in \mathbb{R}^{3,2}$. We may write $\mathfrak{r}$ in terms of the basis $\{\mathfrak{f},\mathfrak{t},\mathfrak{f}',\mathfrak{q},\mathfrak{p}\}$ as 
\[ \mathfrak{r} = \alpha\mathfrak{f} +\beta\mathfrak{t} +\gamma \mathfrak{f}' + \delta\mathfrak{q} +\epsilon\mathfrak{p},\]
for some functions $\alpha, \beta, \gamma, \delta, \epsilon$. 
$\mathfrak{r}'=0$ implies that
\begin{gather*}
	\alpha'-\gamma\kappa = 0, \quad \beta' + \gamma \chi k =0, \quad \alpha - \beta k +\gamma' = 0,\\
	\gamma + \delta' = 0, \quad \gamma k- \epsilon' =0. 
\end{gather*}

The condition $(\mathfrak{t} + \frac{k}{2}\mathfrak{f},\mathfrak{r})=0$ implies that $\epsilon +\frac{k}{2}\delta = 0$. It thus follows that $k\delta' = k'\delta$. Hence, if $k\not\equiv 0$, we may write $\delta = Ak$ for some constant $A$. If $A = 0$ then it follows that $\gamma = \delta = \epsilon =0$ and thus $\mathfrak{r}\in \Gamma f$. Hence, $C$ is circular. Otherwise we may assume, without loss of generality, that $A=1$. It then follows that $\epsilon = - \frac{k^{2}}{2}$, $\gamma = - \kappa'$ and that $\alpha = -k\kappa - \lambda$ and $\beta =  \frac{k^{2}}{2}\chi +\mu$ for some constants $\lambda$ and $\mu$. Thus $\mathfrak{r}$ has the form~\eqref{eqn:c} and it follows that $\mathfrak{f}$ is a constrained elastic curve satisfying~\eqref{eqn:elsecond}.

\begin{proposition}
\label{prop:elasticcomplex}
$\mathfrak{f}$ is a constrained elastic curve if and only if $(\mathfrak{t} + \frac{k}{2}\mathfrak{f},\mathfrak{r})=0$ for some non-zero constant $\mathfrak{r}\in \mathbb{R}^{3,2}$.
\end{proposition}

Hence, we obtain the following geometric characterisation for constrained elastic curves in space forms.
\begin{theorem}\label{thm:elasticcomplex}
	A curve with curvature $k$ is constrained elastic if and only if the enveloped congruence  of circles with geodesic curvature $\frac{k}{2}$ takes values in a fixed linear circle complex.
\end{theorem}

\begin{remark}
Notice from~\eqref{eqn:c} that when $\lambda =0$ one has that $\mathfrak{r} \perp \mathfrak{q}$, whereas when $\mu =0$ one has that $\mathfrak{r}\perp \mathfrak{p}$. Thus, by Remark~\ref{rem:spcomplex}, for elastic curves in Euclidean 2-space with $k'^{2}+ \frac{k^4}{4} +\mu\, k^{2} >0$ (so that $\mathfrak{r}$ is spacelike), we have that the circle congruence $\langle \mathfrak{t} + \frac{k}{2}\mathfrak{f}\rangle$ intersects a fixed line at a constant angle $\theta$. Moreover for free elastic curves $\theta = \frac{\pi}{2}$. See Figure \ref{fig:elastic} for some examples of this principle.
\end{remark}

\begin{figure}
	\centering
		\includegraphics[width=0.5\textwidth]{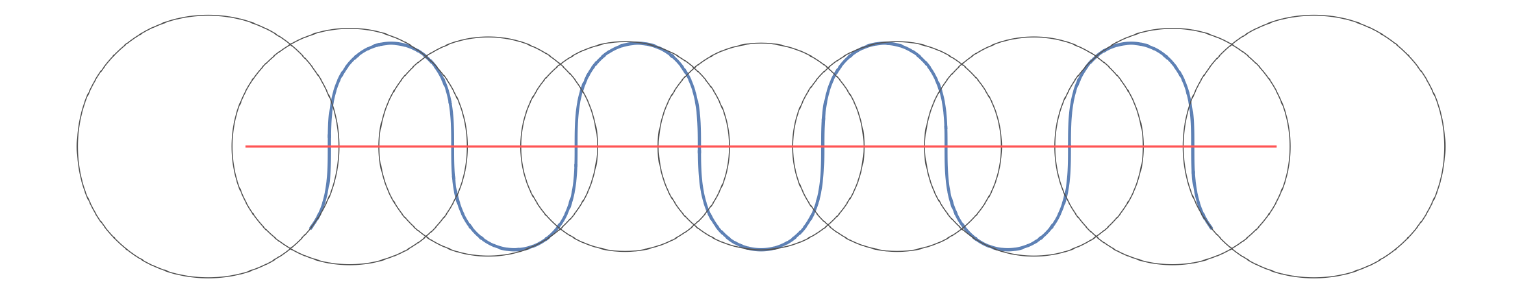}
	
	\centering
	\begin{minipage}{0.51\textwidth}
		\centering
		\includegraphics[width=0.95\linewidth]{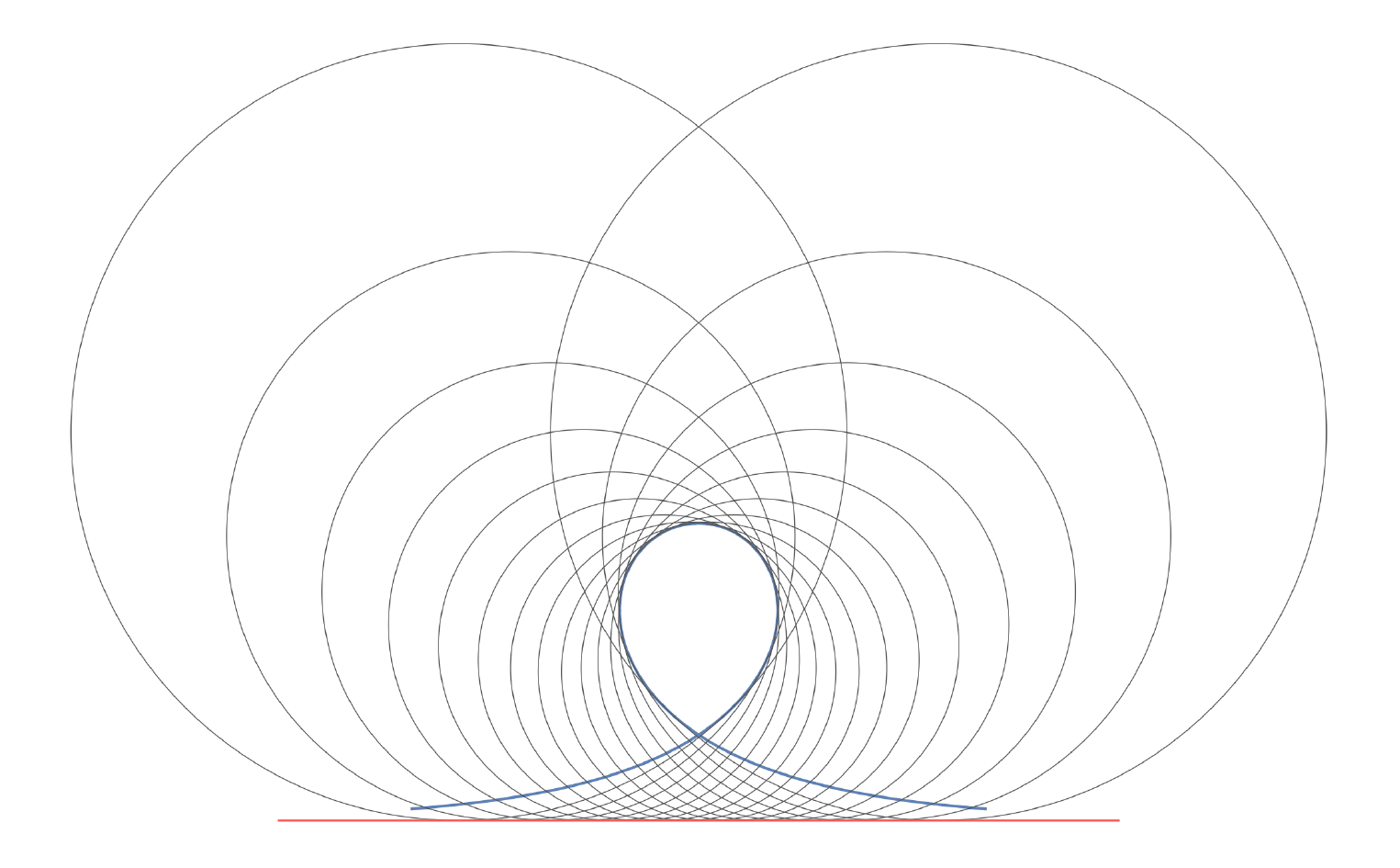}
	\end{minipage}
	\begin{minipage}{0.45\textwidth}
		\centering
		\includegraphics[width=0.8\linewidth]{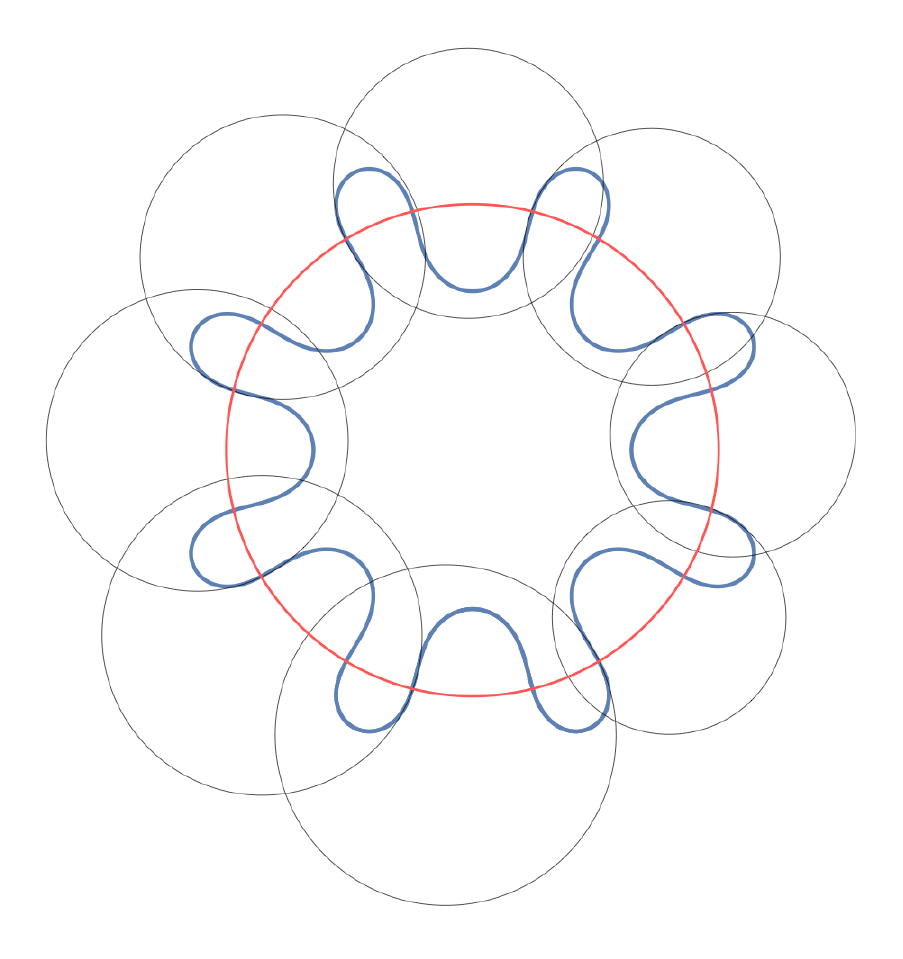}
	\end{minipage}
	
	\caption{Various constrained elastic curves (in blue) and the enveloped congruence of circles with geodesic curvature $\frac{k}{2}$ (in grey) intersecting a fixed circle or line (in red) at a fixed angle.}
	\label{fig:elastic}
\end{figure}

Using the result of Proposition~\ref{prop:elasticcomplex}, we are now able to prove another characterisation of constrained elastic curves in terms of linear conserved quantities.
\begin{theorem}
\label{thm:constrained}
$d+t\xi$ admits a $3$-dimensional space of linear conserved quantities $W$ for some polarisation $[\xi]$ of $C$ if and only if $C$ projects to a constrained elastic curve in an appropriate space form. The point sphere complex $\mathfrak{p}$ and space form vector $\mathfrak{q}$ of this space form satisfy $\mathfrak{p},\mathfrak{q}\in W(0)$. 
\end{theorem}

\begin{proof}
Suppose that $C$ projects to a constrained elastic curve in a space form $\mathfrak{Q}^{2}$ with point sphere complex $\mathfrak{p}$ and space form vector $\mathfrak{q}$. Then by Proposition~\ref{prop:elasticcomplex} we have that $(\mathfrak{t} + \frac{k}{2}\mathfrak{f},\mathfrak{r})=0$ for some non-zero constant $\mathfrak{r}\in \mathbb{R}^{3,2}$. Using the form of~\eqref{eqn:c}, we have that $(\mathfrak{f},\mathfrak{r})=-k$ and $(\mathfrak{t},\mathfrak{r})=\frac{k^{2} }{2}$. Now consider the arclength polarisation $[\xi]$ with $\xi = -\mathfrak{f}\wedge d\mathfrak{f}$. Define $r(t):= \mathfrak{r} + t(2\mathfrak{t} + k\mathfrak{f})$. Then one can check that $r(t)$ is a linear conserved quantity of $d+t\xi$. Thus we have a 3-dimensional space of linear conserved quantities $W = \langle p,q,r\rangle $ of $d+t\xi$, where $p(t)= \mathfrak{p}$ and $q(t)= \mathfrak{q}-t\mathfrak{f}$. 

Conversely, suppose that $d+t\xi$ admits a $3$-dimensional space of linear conserved quantities $W$. Since circular curves are constrained elastic curves we shall assume, without loss of generality, that $C$ is non-circular. By Lemma~\ref{lem:arcpol2}, $[\xi]$ is the arclength polarisation of $C$ with respect to $\mathfrak{p}:= p(0)$ and $\mathfrak{q}:=q(0)$, for some $p,q \in W$. Denoting by $\mathfrak{f}$ and $\mathfrak{t}$ the point sphere map and tangent plane congruence of $C$ with respect to $\mathfrak{p},\mathfrak{q}$, we have that $\xi = -\mathfrak{f}\wedge d\mathfrak{f}$. Now let $r\in W$ such that $r\not\in \langle p,q\rangle$ and write $r(t) = \mathfrak{r} + t(a \mathfrak{f} + b\mathfrak{t})$ for some $\mathfrak{r}\in \mathbb{R}^{3,2}$ and functions $a,b$. The condition that $r$ is a linear conserved quantity amounts to
\[ 0 = da \mathfrak{f} + a d\mathfrak{f} + db \mathfrak{t} + bd\mathfrak{t} - (\mathfrak{f},\mathfrak{r})d\mathfrak{f} + (d\mathfrak{f},\mathfrak{r})\mathfrak{f}.\]
Thus $db=0$, $da +(d\mathfrak{f},\mathfrak{r}) = 0$ and $a - bk - (\mathfrak{f},\mathfrak{r}) =0$. Hence 
\[ (\mathfrak{f},\mathfrak{r}) = \lambda k + \mu\]
for some constants $\lambda,\mu$. It then follows that $(d\mathfrak{t},\mathfrak{r}) = - \lambda kdk$ and thus $(\mathfrak{t},\mathfrak{r}) = -\frac{\lambda}{2} k^{2} + \zeta$ for some constant $\zeta$. Thus
\[ (\mathfrak{t} + \frac{k}{2}\mathfrak{f}, \mathfrak{r})  = \zeta + \mu\frac{k}{2}.\]
Defining $\tilde{\mathfrak{r}}:= \mathfrak{r} + \zeta \mathfrak{p} + \mu \mathfrak{q}$, we have that 
\[ (\mathfrak{t} + \frac{k}{2}\mathfrak{f}, \tilde{\mathfrak{r}}) = 0.\]
Thus by Proposition~\ref{prop:elasticcomplex} we have that $\mathfrak{f}$ is a constrained elastic curve. 
\end{proof}

In light of Theorem~\ref{thm:constrained} we make the following definition:

\begin{definition}
\label{def:conelas}
Let $W_{0}$ be a $3$-dimensional subspace of $\mathbb{R}^{3,2}$. We say that $C:I\to\mathcal{Z}$ is a \textit{constrained elastic curve with respect to $W_{0}$} if there exists a polorarisation $[\xi]$ of $C$ such that $d+t\xi$ admits a $3$-dimensional space $W$ of linear conserved quantities with $W(0)=W_{0}$. 
\end{definition}

The following proposition shows that, for non-circular curves, the spaces of linear conserved quantities of $d+t\xi$ has maximal dimension $3$: 

\begin{proposition}
\label{prop:4lcq}
If $d+t\xi$ admits a $4$-dimensional space of linear conserved quantities then $C$ is circular. 
\end{proposition}
\begin{proof}
Suppose that $C$ admits a $4$-dimensional space of linear conserved quantities. If $C$ is non-circular, then as in the proof of Theorem~\ref{thm:constrained}, we can show that there exists two linearly independent vectors $\mathfrak{r}_{1}, \mathfrak{r}_{2}\in \mathbb{R}^{3,2}$ such that 
\[ (\mathfrak{t} + \frac{k}{2}\mathfrak{f}, \mathfrak{r}_{1}) =  (\mathfrak{t} + \frac{k}{2}\mathfrak{f}, \mathfrak{r}_{2}) =0\]
for some space form projection $\mathfrak{f}$ of $C$. Thus $\tilde{c}:= \langle \mathfrak{t} + \frac{k}{2}\mathfrak{f}\rangle$ lives in a constant $3$-dimensional space $W_{0}:= \langle \mathfrak{r}_{1},\mathfrak{r}_{2}\rangle^{\perp}$. But then by Lemma~\ref{lem:W0s}, $C$ is circular, contradicting our assumption that $C$ is non-circular. 
\end{proof}

\section{Surfaces with spherical curvature lines}\label{section4}
In this section we shall be considering the case that $n=2$, and thus Legendre maps $f:\Sigma\to\mathcal{Z}$ project to fronts in space forms, that is, surfaces with admissible singularities. We will investigate the case that $f$ projects to fronts with one or two families of spherical curvature lines, following the approach of Blaschke~\cite[\S 88]{B1929}. We shall begin this section by recalling some facts about curvature lines of Legendre maps. 

Let $f:\Sigma\to\mathcal{Z}$ be a Legendre map. Then, since $f^{(1)}\le f^{\perp}$, we can define a tensor 
\[ \beta: T\Sigma\to \Hom(f,f^{\perp}/f)\quad X\mapsto  (\sigma\mapsto d_{X}\sigma\, mod\, f).\]
Moreover, since $f$ is a Legendre map, we have the condition $\ker \beta = \{0\}$. We can then define a tensor $\confstr \in \Gamma (S^{2}T\Sigma\otimes (\wedge^{2}f)^{*}\otimes \wedge^{2}(f^{\perp}/f))$ by 
\[ \confstr(X,Y)\xi_{1}\wedge \xi_{2} = \frac{1}{2}(\beta(X)\xi_{1}\wedge \beta(Y)\xi_{2} + \beta(Y) \xi_{1}\wedge \beta(X)\xi_{2}),\]
where $X,Y\in \Gamma T\Sigma$ and $\xi_{1},\xi_{2}\in \Gamma f$. Away from points where $\confstr$ vanishes, we can identify $\confstr$ with an indefinite conformal structure on $\Sigma$ whose null vectors are curvature directions of $f$ (see~\cite[Section 2.4.1]{P2020}). Thus, if $T\le T\Sigma$ is a rank $1$ distribution satisfying $\confstr(T,T)=0$, we have that the leaves of $T$ are curvature lines of $f$. We shall denote by $f_{T}$ the set of sections of $f$ and derivatives of sections of $f$ along $T$. 

\begin{lemma}
\label{lem:Tcurv}
Let $T\le T\Sigma$ be a rank $1$ distribution. Then:
\begin{enumerate}
\item $\confstr(T,T)=0$ if and only if $f_{T}$ is a rank $3$ subbundle of $f^{(1)}$. \label{item:fTrank3}
\item If $\confstr(T,T)=0$ then there exists a complementary rank $1$ distribution $\tilde{T}\le T\Sigma$ such that $\confstr(\tilde{T},\tilde{T})=0$ and $f_{T}\perp f_{\tilde{T}}$. 
\end{enumerate}
\end{lemma}
\begin{proof}
Since $f$ is a Legendre map, one has that $\dim f_{T}(x)\ge 3$ for all $x\in \Sigma$. On the other hand, by definition of $\confstr$, it is straightforward to see that $\confstr(T,T)=0$ if and only if $(f_{T}/f)\wedge (f_{T}/f) =\{0\}$, i.e., $f_{T}/f$ has rank $1$. This proves~\eqref{item:fTrank3}. 

Suppose that $\confstr(T,T)=0$ and thus $f_{T}$ has rank $3$, by~\eqref{item:fTrank3}. Then $s_{T}:= \ker\beta(T)$ is a rank $1$ subbundle of $f$. Fix a nowhere-zero section $\sigma_{T}\in \Gamma s_{T}$. 

Since $f$ is a Legendre map, we may find a section $\sigma\in \Gamma f$ such that $\beta\sigma$ is an immersion. It then follows that $(d\sigma,d\sigma)$ is a positive definite metric on $\Sigma$. Let $\tilde{T}\le T\Sigma$ be orthogonal to $T$ with respect to this metric. We have that 
\[ f_{\tilde{T}} = f + \langle d_{\tilde{X}}\sigma, d_{\tilde{X}}\sigma_{T}\rangle\]
for some nowhere-zero section $\tilde{X}\in \Gamma \tilde{T}$. On the other hand, since $f_{T}$ is rank $3$, we have that 
\[ f_{T} = f\oplus \langle d_{X}\sigma\rangle, \]
for some nowhere-zero section $X\in \Gamma T$. By construction of $\tilde{T}$, we have that $d_{\tilde{X}}\sigma \perp f_{T}$. Using that $f^{(1)}\le f^{\perp}$, we have that 
\[ (d_{X}\sigma, d_{\tilde{X}}\sigma_{T}) = (d_{\tilde{X}}\sigma, d_{X}\sigma_{T})=0,\]
since $\sigma_{T} \in \Gamma \ker \beta(T)$. Hence, $f_{\tilde{T}}\perp f_{T}$ and it follows that $f_{\tilde{T}}$ has rank $3$. Applying~\eqref{item:fTrank3} again implies that $\confstr(\tilde{T},\tilde{T})=0$. 
\end{proof}

\subsection{Spherical curvature lines}

In this paper we are interested in surfaces in space forms with one or two families of spherical curvature lines. In order to express this notion, we break symmetry to a conformal subgeometry. Let $\mathfrak{p}\in \mathbb{R}^{4,2}$ be a timelike point sphere complex for a conformal Riemannian geometry $\langle \mathfrak{p}\rangle^{\perp}\cong \mathbb{R}^{4,1}$. Let $\Lambda:= f\cap\langle \mathfrak{p}\rangle^{\perp}$ be the point sphere map of $f$ with respect to this conformal geometry. Suppose that $T\le T\Sigma$ is a rank $1$ distribution satisfying $\confstr(T,T)=0$ and thus the leaves of $T$ are a family of curvature lines of $\Lambda$. The condition that the leaves of $T$ be spherical curvature lines then amounts to $\Lambda \perp \nu$ for some smooth map $\nu:\Sigma\to \mathbb{P}(\mathcal{L})$ that is constant along the leaves of $T$. Recall from Subsection~\ref{subsec:symbreak} that we obtain the conformal representative of $\nu$ from $S:= \pi_{\mathfrak{p}}(\nu)$, where $\pi_{\mathfrak{p}}$ denotes the projection of $\mathbb{R}^{4,2}$ onto $\langle \mathfrak{p}\rangle^{\perp}$. 

\begin{remark}
\label{rem:transversal}
Throughout this paper, we will only be considering the case that the spheres defined by $\nu$ intersect $\Lambda$ transversally. In particular, this implies that $\nu\cap f^{\perp} = \{0\}$. 
\end{remark}

\begin{theorem}
\label{thm:Lsph}
There exists a spacelike rank $1$ subbundle $L$ of $f^{\perp}$ that is constant along the leaves of a rank $1$ distribution $T\le T\Sigma$ if and only if, where the projection immerses, $f$ projects to a surface with spherical curvature lines along the leaves of $T$ in some (and in fact, any) conformal Riemannian geometry.
\end{theorem}
\begin{proof}
Suppose that a spacelike rank $1$ subbundle $L\le f^{\perp}$ is constant along the leaves of $T\le T\Sigma$. Then $f_{T}\perp L$, and thus $f_{T}$ is a rank $3$ subbundle of $f^{(1)}$. From Lemma~\ref{lem:Tcurv} we have that the leaves of $T$ are curvature lines of $f$. Choose a timelike point sphere complex $\mathfrak{p}\in \mathbb{R}^{4,2}$. Then $W:= L\oplus\langle \mathfrak{p}\rangle$ is rank $2$ bundle with signature $(1,1)$. Let $\nu^{\pm}$ denote the two rank $1$ null subbundles of $W$. Then $\nu^{\pm}$ are constant along the leaves of $T$. Moreover, defining $\Lambda:= f\cap\langle\mathfrak{p}\rangle^{\perp}$ to be the point sphere map of $f$ with respect to $\mathfrak{p}$, we have that $\Lambda\perp \nu^{\pm}$. Thus $\Lambda$ has spherical curvature lines along the leaves of $T$. 

Conversely, assume that the point sphere map $\Lambda$ with respect to some conformal geometry defined by timelike point sphere complex $\mathfrak{p}\in \mathbb{R}^{4,2}$ immerses and has spherical curvature lines along the leaves of a rank $1$ distribution $T\le T\Sigma$. Thus $\Lambda \perp \nu$ for some $\nu:\Sigma\to \mathbb{P}(\mathcal{L})$ that is constant along the leaves of $T$. Then let $W:= \nu \oplus \langle\mathfrak{p}\rangle$. Now $W\perp \Lambda$ and, since $\mathfrak{p}$ is timelike, we have that $L:= f^{\perp}\cap W$ is a rank $1$ spacelike subbundle of $f^{\perp}$. Since the leaves of $T$ are curvature lines and $\Lambda$ is an immersion, we have that $f_{T} = f + \Lambda_{T}$, where $\Lambda_{T}$ denotes the set of sections of $\Lambda$ and derivatives of sections of $\Lambda$ along $T$. Since $\Lambda\perp W$ and $W$ is constant along $T$, we learn that $\Lambda_{T}\perp W$. Hence $L_{T}\perp f$ and, since $L_{T}\le W$, we deduce that $L$ is constant along the leaves of $T$. 
\end{proof}

\begin{remark}
Suppose that $\mathfrak{p}$ is a timelike point sphere complex for a Riemannian conformal geometry $\langle \mathfrak{p}\rangle^{\perp}$. Then from Subsection~\ref{subsec:linsphcom}, we have that $S:= \pi_{\mathfrak{p}}(L)$ yields the conformal representative of the spheres containing the curvature lines of $f$ in Theorem~\ref{thm:Lsph}. 
\end{remark}

\begin{remark}
Notice that Theorem~\ref{thm:Lsph} implies that even though we broke symmetry to a conformal geometry to express the condition of having spherical curvature lines, it is in fact a Lie invariant notion. 
\end{remark}

\subsection{Surfaces with spherical curvature lines via Lie spherical evolutions}
\label{subsec:sphevo}

Denote by $\mathcal{S}$ the Grassmannian of $1$-dimensional spacelike subspaces of $\mathbb{R}^{4,2}$. Then rank $1$ spacelike subbundles of $\underline{\mathbb{R}}^{4,2}$ can be viewed as maps $\Sigma\to \mathcal{S}$. Given a map $L:\Sigma\to \mathcal{S}$, we may split the trivial bundle as
	\[
		\Sigma \times \mathbb{R}^{4,2} = L \oplus L^\perp,
	\]
inducing a splitting of the trivial connection
	\[
		d = \mathcal{D}^{L} + \mathcal{N}^{L}
	\]
where $\mathcal{D}^{L}$ is the sum of the induced connections on $L$ and $L^\perp$, and $\mathcal{N}^{L} \in \Omega^1 (L \wedge L^\perp)$.

Suppose that $f : \Sigma \to \mathcal{Z}$ admits spherical curvature lines along the leaves of a rank $1$ distribution $T\le T\Sigma$. Thus, by Theorem~\ref{thm:Lsph}, there exists a rank $1$ spacelike subbundle $L\le f^{\perp}$ that is constant along the leaves of $T$. We then have that $\mathcal{N}^{L}(T)=0$, from which one deduces that $\mathcal{D}^{L}$ is a flat metric connection. By Lemma~\ref{lem:Tcurv} there exists a complementary rank $1$ null subbundle $\tilde{T}$ of the conformal structure $\confstr$ with $f_{T}\perp f_{\tilde{T}}$. Since $f\perp L$, we have that $f_{T}\perp L$ and thus $f_{\tilde{T}} = f\oplus L$. Hence, $f$ is a parallel subbundle of $\mathcal{D}^{L}|_{\tilde{T}}$. We thus arrive at the following lemma: 

\begin{lemma}
\label{lem:DLflat}
$\mathcal{D}^{L}$ is a flat metric connection and $f$ is a parallel subbundle of $\mathcal{D}^{L}|_{\tilde{T}}$.
\end{lemma}

Locally we may restrict $f$ to a rectangular domain $\Sigma = I\times \tilde{I}$ where the parameter lines along $I$ are the leaves of $T$ and the parameter lines of $\tilde{I}$ are the leaves of $\tilde{T}$. Thus, $f$ is parametrised by curvature lines $(u,v)$ with $u\in I$ and $v\in \tilde{I}$. Since $L$ is constant along the leaves of $T$, we may view it as a map from $\tilde{I}\to \mathcal{S}$. 

Fixing a point $v_{0}\in \tilde{I}$, there exists an orthogonal trivialising gauge transformation $A : \tilde{I} \to \Ortho(4,2)$ of $\mathcal{D}^{L}$ such that
	\begin{equation}
	\label{eqn:sphevo}
		A \cdot \mathcal{D}^{L} = d \quad \text{and}\quad A(v_{0}) = id. 
	\end{equation}
Note that the construction of $A$ is completely independent of $f$ and we can make the following definition for an arbitrary map $L:\tilde{I}\to \mathcal{S}$: 

\begin{definition}\label{def:sem}
Given a map $L:\tilde{I}\to \mathcal{S}$ and $v_{0}\in\tilde{I}$, we call $A : \tilde{I} \to \Ortho(4,2)$ satisfying~\eqref{eqn:sphevo} the \textit{spherical evolution map along $L$ based at $v_{0}$}.   
\end{definition}

Clearly $L$ is a parallel subbundle of $\mathcal{D}^{L}$. Thus, with $L_{0}:= L(v_{0})$, we have that $L = A^{-1} L_{0}$. 
Since $f$ is a parallel subbundle of $\mathcal{D}^{L}|_{\tilde{T}}$ and $f\perp L$, there exists $C:I\to \mathcal{Z}$ with $C\perp L_{0}$ such that $f = A^{-1} \, C$. Since $f$ is a Legendre map, we have that $C$ is a Legendre curve when viewed as a map into $L_{0}^{\perp}\cong \mathbb{R}^{3,2}$. 

Conversely, suppose that $C:I\to \mathcal{Z}$ is a Legendre curve in $L_{0}^{\perp}$ and define $f:= A^{-1}C:I\times \tilde{I}\to\mathcal{Z}$. We have that $f_{T} = A^{-1}C^{(1)}$ and thus $f_{T}$ is a rank $3$ subbundle of $f^{\perp}$. Now $f\perp L$ and $f$ is parallel for $D^{L}|_{\tilde{T}}$, thus $f_{\tilde{T}}\le f\oplus L\le f^{\perp}$. In order that $f$ be a Legendre map, we need to have that $f_{\tilde{T}} = f\oplus L$. Since $A\cdot \mathcal{D}^{L} = d$ and $\mathcal{D}^{L} = d - \mathcal{N}^{L}$, it follows that 
\begin{equation} 
\label{eqn:A-1gauge}
A^{-1}\cdot (d + Ad_{A}\mathcal{N}^{L}) = d.
\end{equation}
One then deduces that 
\[f_{\tilde{T}} = f + A^{-1}(Ad_{A}\mathcal{N}^{L}(\tilde{T}))C.\]
Therefore, in order that $f_{\tilde{T}}$ has rank $3$, we require that $(Ad_{A}\mathcal{N}^{L})|_{C}$ nowhere vanishes\footnote{This condition is akin to the regularity condition required to construct Monge surfaces, see for example~\cite{BG2018}.}. Once this condition is established, it follows from Lemma~\ref{lem:Tcurv} that the parameter lines of $I$ and $\tilde{I}$ are curvature lines of $f$, and from Theorem~\ref{thm:Lsph} it follows that the $I$ parameter lines are spherical curvature lines. 

We have thus arrived at the following theorem, illustrated in Figure \ref{fig:first}: 

\begin{theorem}\label{thm:A}
Let $L:\tilde{I}\to \mathcal{S}$ be an immersion, set $L_{0}:= L(v_{0})$ for some $v_{0}\in \tilde{I}$ and denote by $A$ the spherical evolution map along $L$ based at $v_{0}$. Let $C:I\to \mathcal{Z}$ be a Legendre curve in $L_{0}^{\perp}\cong \mathbb{R}^{3,2}$ such that $(Ad_{A}\mathcal{N}^{L})|_{C}$ is nowhere zero. Then 
\[ f:I\times \tilde{I}\to \mathcal{Z}, \quad f(u,v) = A^{-1}(v)C(u)\]
is a Legendre map parametrised by curvature lines for which the $u$-parameter lines are spherical. 

Conversely, any surface with one family of spherical curvature lines may locally be written in this way. 
\end{theorem}


\begin{figure}
	\centering
	\includegraphics[width=0.98\textwidth]{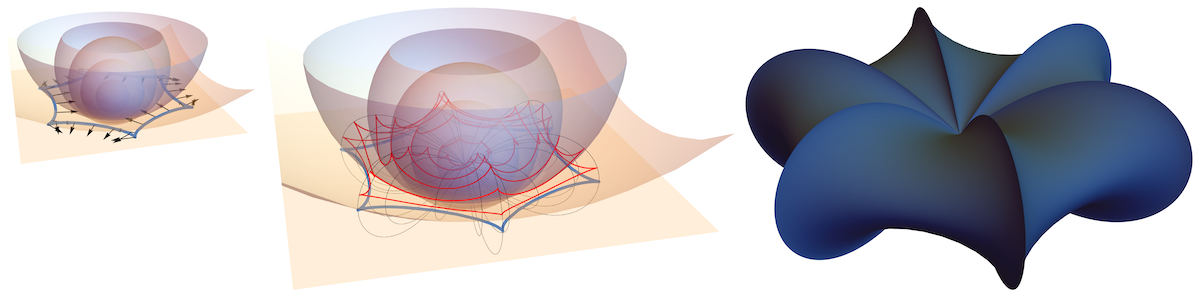}
	\includegraphics[width=0.98\textwidth]{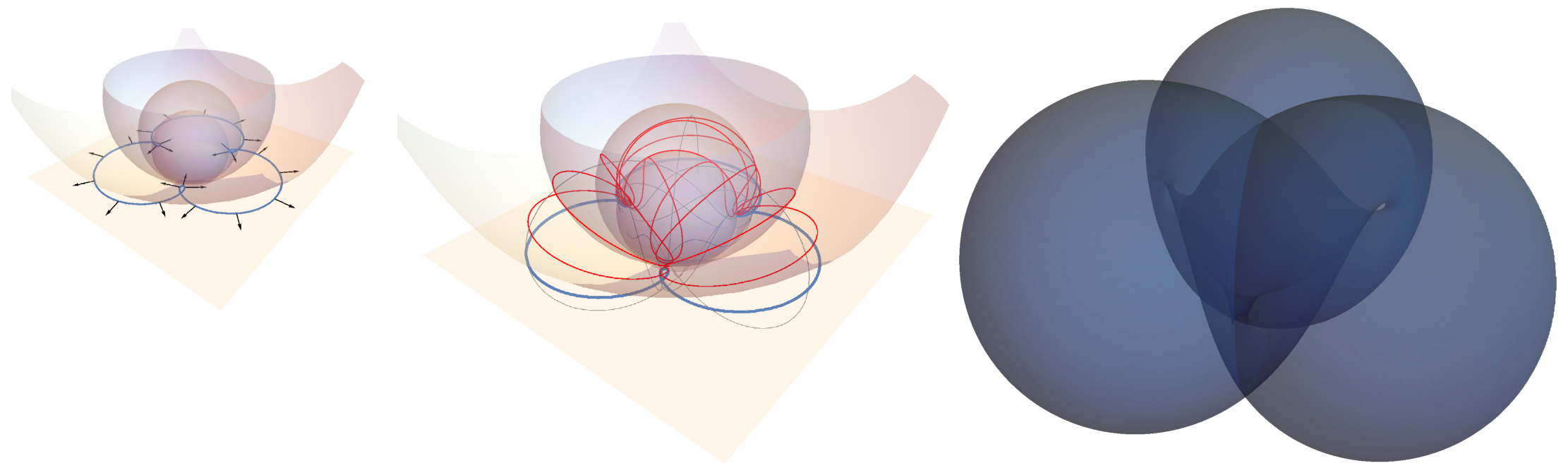}
	\caption{Given a curve of spheres with prescribed angles of intersection, i.e., a curve of elliptic sphere complexes, one can obtain the spherical evolution map and evolve an initial appropriate Legendre curve.
		In these examples, one obtains either a well-known constant negative Gaussian curvature torus (top) (see, for example, \cite{E1868, melko_integrable_1994, ura_constant_2018}), or the Wente torus (bottom) (see, for example, \cite{A1987, W1987, W1986}).}
	\label{fig:first}
\end{figure}

\begin{remark}
For a different choice of initial point $\tilde{v}_{0}\in \tilde{I}$, the spherical evolution map $\tilde{A}$ along $L$ based at $\tilde{v}_{0}$ differs from $A$ by $\tilde{A} = BA$ for some constant $B\in \Ortho(4,2)$. We then have that $\tilde{C}:=BC \perp \tilde{L}_{0}$ with $\tilde{L}_{0} = L(\tilde{v}_{0})$ and the resulting Legendre map 
\[ \tilde{f} = \tilde{A}^{-1}\tilde{C} = A^{-1}B^{-1}BC = A^{-1}C = f.\]
In this sense, the construction in Theorem~\ref{thm:A} is independent of the choice of $v_{0}$, 
\end{remark}

\begin{remark}
As noted by Blaschke~\cite{B1929}, Theorem~\ref{thm:A} implies that if one family of curvature lines is spherical then any two members of this family are related by a Lie sphere transformation. 
\end{remark}

\begin{remark}
Recall from Remark~\ref{rem:transversal} that we excluded the case that the spheres containing curvature lines are tangential to the surface. Without this restriction, $L$ could become lightlike, resulting in singular points of the corresponding spherical evolution maps $A$.  
\end{remark}

\begin{remark}
Suppose that $c$ is the curvature circle congruence of $C$ (viewed as a Legendre curve in $L_{0}^{\perp}\cong \mathbb{R}^{3,2}$). As a map into $\mathbb{R}^{4,2}$, $c$ yields a $1$-parameter family of curvature spheres along the initial spherical curve of $f$. One then obtains a curvature sphere congruence by parallel transporting $c$ along $L$, i.e., $s:= A^{-1}c$ is a curvature sphere congruence of $f=A^{-1}C$. 
\end{remark}

\subsection{Osculating sphere complexes}
In this subsection we shall consider the case that $f:\Sigma\to\mathcal{Z}$ is an umbilic-free Legendre map. We then recover the results of~\cite[\S 88]{B1929}, defining the osculating sphere complexes of such an $f$ and characterising the presence of spherical curvature lines in terms of these. 

Assume that $f:\Sigma\to\mathcal{Z}$ is an umbilic-free Legendre map i.e., there are two curvature sphere congruences $s_{1},s_{2}\le f$ with $s_{1}\cap s_{2}=\{0\}$. Writing the corresponding curvature subbundles as $T_{1}:= T_{s_{1}}$ and $T_{2}:= T_{s_{2}}$, we have that 
\[ d_{X}\sigma_{1}, d_{Y}\sigma_{2}\in \Gamma f\]
for any $\sigma_{1}\in \Gamma s_{1}$, $\sigma_{2}\in \Gamma s_{2}$, $X\in \Gamma T_{1}$ and $Y\in \Gamma T_{2}$. 

If for some $i\in\{1,2\}$, $s_{i}$ is constant along the leaves of $T_{i}$, i.e., $d_{X_{i}}\sigma_{i} \in \Gamma s_{i}$, then $f$ is a \textit{channel surface} and the curvature lines along $T_{i}$ are circular (see~\cite{PS2018}). Whereas if both of the curvature sphere congruences immerse then we say that $f$ is \textit{regular}. 

\begin{remark}
\label{rem:tubular}
If $s_{i}\perp \mathfrak{q}$ for some non-zero $\mathfrak{q}\in \mathbb{R}^{4,2}$ then $f$ is a channel surface and in fact projects to a tube over a curve in space forms (see~\cite{BHR2012}). 
\end{remark}

For $f_{i}:= f_{T_{i}}$ we have that
\[ f_{1} = f\oplus \langle d_{X}\sigma_{2} \rangle \quad \text{and}\quad  f_{2} = f\oplus \langle d_{Y}\sigma_{1} \rangle \]
for non-zero sections $\sigma_{1}\in \Gamma s_{1}$, $\sigma_{2}\in \Gamma s_{2}$, $X\in \Gamma T_{1}$ and $Y\in \Gamma T_{2}$. These form rank 3 subbundles with $f_{1}\perp f_{2}$ and we may write the derived bundle as $f^{(1)} = f_{1}+f_{2}$. Moreover the rank 1 bundles $f_{i}/f$ inherit a positive definite metric from $\mathbb{R}^{4,2}$. 

We may split the trivial bundle as 
	\[
		\underline{\mathbb{R}}^{4,2} = S_1 \oplus_\perp S_2
	\]
where
\begin{equation}
\label{eqn:liecyc}
		S_1 = \langle \sigma_1, d_Y \sigma_1, d_Y d_Y \sigma_1 \rangle \quad\text{and}\quad
		 S_2 = \langle \sigma_2, d_X \sigma_2,d_X d_X \sigma_2 \rangle
\end{equation}
for $\sigma _1 \in \Gamma s_1$, $\sigma_2 \in \Gamma s_2$, $X \in \Gamma T_1$, and $Y \in \Gamma T_2$. This is called the \textit{Lie cyclide splitting}. Geometrically this represents a congruence of Dupin cyclides that make most contact with $f$ along its curvature lines. We may then split the trivial connection as
	\[
		d = \mathcal{D} + \mathcal{N}
	\]
where $\mathcal{D}$ is the sum of the induced connections on $S_1$ and $S_2$ and $\mathcal{N} \in \Omega^{1}(S_{1}\wedge S_{2})$.
By~\cite[Corollary 3.6]{PS2018}, a Legendre map parametrises a channel surface if and only if the Lie cyclides are constant along the leaves of one of the curvature subbundles $T_{i}$, i.e., $\mathcal{N}(T_{i})=0$. 

If $f$ is regular then, since the curvature spheres immerse, we have that $\mathcal{N}(T_{1})s_{1} = s_{2} $ and $\mathcal{N}(T_{2})s_{2}= s_{1}$. By~\eqref{eqn:liecyc} we have that
	\[
		d_Y \sigma_1, d_Y d_Y \sigma_1 \in \Gamma S_{1}\quad\text{and}\quad d_X \sigma_2, d_X d_X \sigma_2 \in \Gamma S_{2}.
	\]
Since $S_{1}\cap f^{\perp} =\langle \sigma_{1},d_{Y}\sigma_{1}\rangle$ and $S_{2}\cap f^{\perp} = \langle \sigma_{2},d_{X}\sigma_{2}\rangle$, it follows that 
\begin{equation}
\label{eqn:kerN}
 \ker(\mathcal{N}(T_{2})|_{S_{1}}) = S_{1}\cap f^{\perp} \quad \text{and}\quad  \ker (\mathcal{N}(T_{1}) |_{S_{2}}) = S_{2}\cap f^{\perp}.
 \end{equation}
Hence $\mathcal{N}(T_{1})$ and $\mathcal{N}(T_{2})$ are decomposable subbundles of $S_{1}\wedge S_{2}$. We can then define rank 1 subbundles $L_{i}\le S_{i}$ by
\begin{equation}
\label{eqn:Li}
 L_{1} := \ker(\mathcal{N}(T_{1})) \cap S_{1}\cap f^{\perp} \quad \text{and}\quad L_{2} := \ker(\mathcal{N}(T_{2})) \cap  S_{2}\cap f^{\perp}.
\end{equation}
Since $\mathcal{N}(T_{1})s_{2}=s_{1}$ and $\mathcal{N}(T_{2})s_{1}=s_{2}$, we must have that $L_{i}\cap s_{i} = \{0\}$. Thus the $L_{i}$ are spacelike. \eqref{eqn:kerN} and \eqref{eqn:Li} together imply that the $L_{i}$ are the unique\footnote{If $f$ is a channel surface enveloping the sphere curve $s_{i}$ then the uniqueness of $L_{i}$ fails. However, one of the $L_{i}$ may still be uniquely determined as long as the corresponding curvature sphere $s_{i}$ is an immersion.} subbundles of $S_{i}$ satisfying $\mathcal{N}L_{i}\equiv 0$. We have thus arrived at the following lemma: 

\begin{lemma}
\label{lem:osc}
Suppose that $f$ is regular. Then for each $i\in \{1,2\}$, there exists a unique spacelike rank 1 subbundle $L_{i}\le S_{i}$ such that $\mathcal{N}L_{i} \equiv 0$. Moreover, $f\perp L_{i}$ and $f_{1} = f\oplus L_{2}$ and $f_{2}=f\oplus L_{1}$. 
\end{lemma}

\begin{remark}
\label{rem:formLi}
Suppose that $(u,v)$ are curvature line coordinates for $f$ with $\frac{\partial}{\partial u}\in \Gamma T_{1}$ and $\frac{\partial}{\partial v}\in \Gamma T_{2}$. We may choose lifts $\sigma_{1}\in \Gamma s_{1}$, $\sigma_{2}\in \Gamma s_{2}$ such that 
\[ \sigma_{1,u} = \beta \sigma_{2} \quad \text{and}\quad \sigma_{2,v} = \gamma \sigma_{2}\]
for some smooth functions $\beta$ and $\gamma$. The condition that $f$ is regular is equivalent to assuming that $\beta$ and $\gamma$ are nowhere zero. We then have that  
\[ L_{1} = \langle \beta \sigma_{1,v} -  \beta_{v}\sigma_{1}\rangle\quad \text{and} \quad L_{2} = \langle \gamma \sigma_{2,u}- \gamma_{u}\sigma_{2}\rangle. \]
\end{remark}

Since the $L_{i}$ are spacelike, they form two congruences of elliptic linear sphere complexes $E_{L_{i}}$ and, since $f\perp L_{i}$, we have that $f\subset E_{L_{i}}$, that is, at each point $x\in \Sigma$, the spheres contained in $f(x)$ belong to the linear sphere complex $E_{L_{i}(x)}$. This leads us to the following definition:

\begin{definition}
$L_{1}$ and $L_{2}$ are called the \textit{osculating sphere complexes of $f$}. 
\end{definition}

\begin{remark}
One can show that $L_{i}$ is perpendicular to $f$ and derivatives of sections of $f$ along $T_{i}$ to the third order. We can thus interpret $L_{i}$ as the unique linear sphere complex containing the spheres of $4$ infinitesimally close contact elements of $f$ along the curvature lines $T_{i}$. 
\end{remark}

\begin{lemma}
\label{lem:Ldist}
Suppose that $L\le f^{\perp}$ is a rank $1$ spacelike subbundle of $\underline{\mathbb{R}}^{4,2}$ that is constant along the leaves of a rank $1$ distribution $T\le T\Sigma$. Then either $L = L_{1}$ and $T=T_{1}$ or $L= L_{2}$ and $T=T_{2}$. 
\end{lemma}
\begin{proof}
Let $f_{T}$ denote the set of sections of $f$ and derivatives of sections of $f$ along $T$. Since $L$ is constant along $T$, it follows that $f_{T}\perp L$. Thus $f_{T}$ is a rank $3$ subbundle of $f^{(1)}$. By Lemma~\ref{lem:Tcurv} we have that $T$ is either $T_{1}$ or $T_{2}$. Without loss of generality, assume that $T = T_{1}$. By differentiating $s_{2}$ along $T_{1}$, we deduce that $L\perp S_{2}$. Thus $L\le S_{1}\cap f^{\perp}$ and, since $L$ is constant along $T_{1}$, Lemma~\ref{lem:osc} implies that $L = L_{1}$. 
\end{proof}

Using Lemma~\ref{lem:Ldist} and Theorem~\ref{thm:Lsph}, we can recover the characterisation of Blaschke for umbilic-free surfaces that have spherical curvature lines in terms of the osculating sphere complexes: 

\begin{proposition}
\label{prop:sphLi}
	The leaves of $T_i$ are spherical curvature lines of $f$ if and only if the osculating complex $L_i$ is constant along the leaves of $T_i$.
\end{proposition}

\begin{proof}
	Suppose that $L_i$ is constant along the leaves of $T_i$. Then by Theorem~\ref{thm:Lsph} $f$ has spherical curvature lines along the leaves of $T_{i}$. 
	
	Conversely assume, without loss of generality, that the leaves of $T_i$ are spherical curvature lines of $f$. Then by Theorem~\ref{thm:Lsph}, there exists a rank $1$ spacelike subbundle $L\le f^{\perp}$ that is constant along the leaves of $T_{i}$. By Lemma~\ref{lem:Ldist}, it then follows that $L=L_{i}$. 
\end{proof}

An immediate corollary to Proposition~\ref{prop:sphLi} is the following: 

\begin{corollary}
All of the curvature lines of $f$ are spherical if and only if the osculating complexes $L_{1}$ and $L_{2}$ are constant along the leaves of $T_{1}$ and $T_{2}$, respectively. 
\end{corollary}

\begin{remark}
\label{rem:betagamma}
In terms of the special lifts $\sigma_{1},\sigma_{2}$ of the curvature spheres given in Remark~\ref{rem:formLi}, one can deduce that $f$ has spherical curvature lines along the leaves of $T_{1}$ or $T_{2}$ if and only 
\[ (\ln \beta)_{uv} - \beta\gamma =0 \quad\text{or}\quad (\ln \gamma)_{uv} - \beta\gamma =0,\]
respectively.  
\end{remark}

Let $l_{i}\in \Gamma L_{i}$ be unit spacelike sections of $L_{i}$. By Lemma~\ref{lem:osc} we have that $l_{i}\perp f_{i}$ and $\mathcal{N}L_{i}\equiv 0$. It thus follows that 
\begin{equation}
\label{eqn:dTili}
d|_{T_{i}}l_{i}\in \Gamma( T_{i}^{*}\otimes s_{i}).
\end{equation}
From Proposition~\ref{prop:sphLi} we then have that the leaves of $T_{i}$ are spherical curvature lines if and only if $d|_{T_{i}}l_{i} = 0$. 

Fix nowhere-zero sections $X\in \Gamma T_{1}$ and $Y\in \Gamma T_{2}$ and define 
\[ \mathcal{H}_{1}:= \langle l_{1}, d_{Y}l_{1}, d_{Y}d_{Y}l_{1}\rangle \quad \text{and}\quad \mathcal{H}_{2}:= \langle l_{2},d_{X}l_{2},d_{X}d_{X}l_{2}\rangle.\]
Note that $\mathcal{H}_{i}$ are independent of choices. Blaschke~\cite{B1929} showed that these are rank $3$ subbundles of $\underline{\mathbb{R}}^{4,2}$ and $\mathcal{H}_{1}\perp \mathcal{H}_{2}$. It is not always true that $\mathcal{H}_{1}\cap \mathcal{H}_{2} \neq \{0\}$. For example, we shall see in the next subsection that if $f$ has two families of planar curvature lines in some Euclidean geometry with space form vector $\mathfrak{q}$, then we have that $\mathfrak{q}\in \mathcal{H}_{1}\cap \mathcal{H}_{2}$.

\begin{definition}
We call $\mathcal{H}_{i}$ the \textit{osculating bundles} of $f$. 
\end{definition}

Clearly if $L_{i}$ is constant along $T_{i}$ then $\mathcal{H}_{i}$ is constant along $T_{i}$. Conversely, if $\mathcal{H}_{i}$ is constant along $T_{i}$ then $d|_{T_{i}}l_{i} \in\Gamma T_{i}^{*}\otimes \mathcal{H}_{i}$. On the other hand, \eqref{eqn:dTili} implies that $d|_{T_{i}}l_{i}\in \Gamma( T_{i}^{*}\otimes s_{i})$. One deduces that we must then have that $d|_{T_{i}}l_{i} = 0$, as otherwise $\mathcal{H}_{i} = S_{i}$ is constant along $T_{i}$ implying that $f$ is a channel surface. We thus recover the following result: 

\begin{proposition}
\label{prop:Hiconst}
Suppose that $f$ is regular. Then the leaves of $T_{i}$ are spherical curvature lines of $f$ if and only if $\mathcal{H}_{i}$ is constant along $T_{i}$. In particular, both families of curvature lines of $f$ are spherical if and only if $\mathcal{H}_{1}$ and $\mathcal{H}_{2}$ are constant.
\end{proposition}

In the next subsection we shall see that surfaces with two families of spherical curvature lines can be characterised in Euclidean space forms by choosing appropriate point sphere complexes and space form vectors in the constant spaces $\mathcal{H}_{1}$ and $\mathcal{H}_{2}$.

\subsection{Symmetry breaking} 
\label{subsec:sphsymbreak}
In this subsection we shall examine some specific cases of spherical curvature lines that belong to subgeometries of Lie sphere geometry. 

Consider the case that $L_{i}\perp \mathfrak{q}$ for some non-zero $\mathfrak{q}\in \mathbb{R}^{4,2}$. It then follows that $d|_{T_{i}}l_{i}\perp \mathfrak{q}$. Since $f$ is regular, Remark~\ref{rem:tubular} implies that $d|_{T_{i}}l_{i} = 0$, and thus the leaves of $T_{i}$ are spherical curvature lines of $f$. Using Remark~\ref{rem:spcomplex} one can deduce the following proposition:

\begin{proposition}
\label{prop:planorth}
The leaves of $T_{i}$ project to planar curvature lines in any space form with space form vector $\mathfrak{q}$ if and only if $L_{i}\perp \mathfrak{q}$ . 

The leaves of $T_{i}$ are orthogonally intersected spherical curvature lines in a Riemannian conformal geometry with timelike point sphere complex $\mathfrak{p}$ if and only if $L_{i}\perp \mathfrak{p}$. 
\end{proposition}

Recall from Theorem~\ref{thm:A} that we may locally write a surface with spherical curvature lines as $f= A^{-1}C:I\times \tilde{I}\to \mathcal{Z}$ where $A$ is the spherical evolution map along $L:\tilde{I}\to \mathcal{S}$ based at $v_{0}\in \tilde{I}$. In the case that $L\perp \mathfrak{q}$ for some non-zero $\mathfrak{q}\in \mathbb{R}^{4,2}$, we have that $A \mathfrak{q} = \mathfrak{q}$. Using Subsection~\ref{subsec:symbreak} we deduce the following corollary: 

\begin{corollary}
\label{cor:evoplanorth}
If $f$ has one family of planar curvature lines, then we may locally write $f=A^{-1}C$ where $A$ is a $1$-parameter family of Laguerre transformations. 

If $f$ has one family of orthogonally intersected spherical curvature lines, then we may locally write $f=A^{-1}C$ where $A$ is a $1$-parameter family of M\"{o}bius transformations.  
\end{corollary} 

Let $\mathfrak{q},\mathfrak{p}$ be the space form vector and point sphere complex for a Euclidean geometry $\mathfrak{Q}^{3}$. From Proposition~\ref{prop:planorth} we deduce that $L_{i}\perp \mathfrak{q},\mathfrak{p}$ if and only if $T_{i}$ are a family of orthogonally intersected planar curvature lines in $\mathfrak{Q}^{3}$. According to~\cite{BG2018} such surfaces are locally Monge surfaces, that is
a surface formed by taking a curve in $\mathbb{R}^{3}$ and parallel transporting an initial curve lying in the normal plane along that curve. This is reflected by the fact that, locally, spherical evolution maps $A$ along $L_{i}$ satisfy $A\mathfrak{p}=\mathfrak{p}$ and $A\mathfrak{q}=\mathfrak{q}$, and are thus isometries of $\mathfrak{Q}^{3}$. 

\begin{proposition}
\label{prop:monge}
$f$ locally projects to a Monge surface in a Euclidean space form $\mathfrak{Q}^{3}$ defined by point sphere complex $\mathfrak{p}$ and space form vector $\mathfrak{q}$ if and only if one of the osculating complexes satisfies $L_{i}\perp \mathfrak{p},\mathfrak{q}$.
\end{proposition}

Blaschke~\cite[\S 88]{B1929} characterised surfaces with two families of spherical curvature lines up to Lie sphere transformation and showed that three cases exist depending on the signature of the constant bundles $\mathcal{H}_{1}$ and $\mathcal{H}_{2}$. These characterisations can be recovered by making appropriate choices of timelike point sphere complex $\mathfrak{p}$ and lightlike space form vector $\mathfrak{q}$ for a Euclidean geometry $\mathfrak{Q}^{3}$ and evaluating the corresponding point sphere map $\mathfrak{f}= f\cap \mathfrak{Q}^{3}$: 
\begin{itemize}
\item If $\mathcal{H}_{1}$ and $\mathcal{H}_{2}$ both have signature $(2,1)$ then we may choose $\mathfrak{p}\in \mathcal{H}_{1}$ and $\mathfrak{q}\in \mathcal{H}_{2}$. By Proposition~\ref{prop:planorth} it follows that $\mathfrak{f}$ has planar curvature lines along $T_{1}$ and orthogonally intersected spherical curvature lines along $T_{2}$. This characterises $\mathfrak{f}$ as a Joachimsthal surface. 

\item If $\mathcal{H}_{1}$ has signature $(3,0)$ and $\mathcal{H}_{2}$ has signature $(1,2)$, then we may choose $\mathfrak{q},\mathfrak{p}\in \mathcal{H}_{2}$. By Proposition~\ref{prop:monge}, we then have that $\mathfrak{f}$ is a Monge surface with planar curvature lines along $T_{1}$. Moreover, there exists a lightlike vector $\mathfrak{o}\in \mathcal{H}_{2}$ that is perpendicular to $\mathfrak{p}$ with $(\mathfrak{o},\mathfrak{q})=-1$. Thus $\mathfrak{o}$ is a point in $\mathfrak{Q}^{3}$ and the planes defined by $L_{1}$ all contain this point. Hence the planes containing the curvature lines along $T_{1}$ envelop a cone.

\item If $\mathcal{H}_{1}$ and $\mathcal{H}_{2}$ are degenerate, then we may choose $\mathfrak{q}\in \mathcal{H}_{1}\cap \mathcal{H}_{2}$. Thus, by Proposition~\ref{prop:planorth}, we have that $\mathfrak{f}$ has planar curvature lines along $T_{1}$ and $T_{2}$. 
\end{itemize}

\begin{remark}
Using the formulation in Remark~\ref{rem:betagamma}, we have that surfaces with two families of spherical curvature lines satisfy
\[ (\ln \beta)_{uv} = (\ln \gamma)_{uv} = \beta\gamma.\]
In~\cite{F2002} this is used to show that such surfaces are Lie applicable. In Section~\ref{sec:lieappsph} we shall recover this result using a curved flat formulation of Lie applicable surfaces. 
\end{remark}

\section{Ribaucour transforms of surfaces with spherical curvature lines}\label{section5}

Classically a pair of surfaces is called a Ribaucour pair if they envelop a common sphere congruence and the curvature lines of both surfaces correspond. We then say that the two surfaces are Ribaucour transforms of each other. In~\cite{BH2006} a Lie geometric characterisation of Ribaucour transforms was given in terms of the flatness of a certain rank $2$ bundle. In~\cite{B1929} it is shown that Ribaucour transforms of channel surfaces have a family of spherical curvature lines. In this section we shall recover this result and show that the converse is true. 

Suppose that $f,\hat{f}:\Sigma\to\mathcal{Z}$ are Legendre maps such that $s_{0}:=f\cap \hat{f}$ is a rank $1$ subbundle of $f$ and $\hat{f}$. Geometrically, $s_{0}$ is a common sphere congruence enveloped by both $f$ 
and $\hat{f}$. From~\cite[Corollary 2.11]{P2020} we have the following characterisation of when $f$ and $\hat{f}$ are a Ribaucour pair: 

\begin{lemma}
\label{lem:shatsflat}
Suppose that $s\le f$ and $\hat{s}\le \hat{f}$ such that $s\cap s_{0} = \hat{s}\cap s_{0} = \{0\}$. Then $f$ and $\hat{f}$ are Ribaucour transforms of each other if and only if the induced connection on $s\oplus \hat{s}$ is flat.
\end{lemma}

Suppose now that $f$ and $\hat{f}$ are umbilic-free. Let $s_{1}$ and $s_{2}$ denote the curvature sphere congruences of $f$ with respective curvature subbundles $T_{1}$ and $T_{2}$, and likewise let $\hat{s}_{1}$, $\hat{s}_{2}$ denote the curvature sphere congruences of $\hat{f}$ with respective curvature subbundles $\hat{T}_{1}$ and $\hat{T}_{2}$. We can then say that $f$ and $\hat{f}$ are Ribaucour transforms of each other if and only if $\hat{T}_{1}=T_{1}$ and $\hat{T}_{2}=T_{2}$. Moreover, from~\cite[Proposition 4.3]{PS2018} we have the following result: 

\begin{lemma}
\label{lem:curvsphrib}
Suppose that $s_{0}$ nowhere coincides with the curvature sphere congruences of $f$ or $\hat{f}$. Then $f$ and $\hat{f}$ are Ribaucour transforms of each other if and only if  
\[ d\hat{\sigma}_{1}(\hat{T}_{2}) \le s_{1}\oplus \hat{s}_{1}\oplus d\sigma_{1}(T_{2}) \quad \text{and} \quad d\hat{\sigma}_{2}(\hat{T}_{1}) \le s_{2}\oplus \hat{s}_{2}\oplus d\sigma_{2}(T_{1}),\]
where $\sigma_{i}\in \Gamma s_{i}$ and $\hat{\sigma}_{i}\in \Gamma \hat{s}_{i}$. 
\end{lemma}

For later use we have the following technical lemma: 
\begin{lemma}
\label{lem:regrib}
Suppose that $f$ is a regular umbilic-free Legendre map and suppose that $\hat{f}:\Sigma\to \mathcal{Z}$ with $\hat{f}^{(1)}\perp \hat{f}$ and $f\cap \hat{f} = s_{0}$ for some rank $1$ null subbundle $s_{0}$. Then $\hat{f}$ is a Legendre map and $s_{0}$ nowhere coincides with the curvature sphere congruences of $f$ or $\hat{f}$. 
\end{lemma}
\begin{proof}
Fix a non-zero section $\sigma_{0}\in \Gamma s_{0}$. Then $d\sigma_{0} \perp f+\hat{f}$. Since $f$ is regular, if $d_{X}\sigma_{0}\in f(x)$ for some $X\in T_{x}\Sigma$, then we would have that $f = s_{0}(x) \oplus \langle d_{X}\sigma_{0}\rangle$. This would imply that $f(x) \le (\hat{f}(x))^{\perp}$ and thus $f(x)=\hat{f}(x)$, contradicting our assumption that $f\cap \hat{f}$ is rank $1$. It thus follows that $d\sigma_{0}(T\Sigma)$ is a spacelike rank $2$ subbundle of $f^{\perp}$. Since $s_{0}\le \hat{f}$, together with the condition $\hat{f}^{(1)}\le \hat{f}^{\perp}$, one deduces that $\hat{f}$ is a Legendre map. 
\end{proof}

\subsection{Ribaucour transforms of surfaces with spherical curvature lines}

Suppose that $f$ is an umbilic-free channel surface. Without loss of generality, assume that $f$ has circular curvature directions along the leaves of $T_{1}$. From~\cite[Proposition 3.4]{PS2018} we have that the curvature sphere congruence $s_{1}$ is constant along the leaves of $T_{1}$. It thus follows that $s_{1}\oplus d\sigma_{1}(T_{2})$ is constant along the leaves of $T_{1}$. 

Now suppose that $\hat{f}$ is a regular umbilic-free Ribaucour transform of $f$ with enveloping sphere congruence $s_{0} = f\cap\hat{f}$. By Lemma~\ref{lem:regrib} we have that $s_{0}$ nowhere coincides with the curvature sphere congruences of $\hat{f}$. We may then apply Lemma~\ref{lem:curvsphrib} to learn that 
\begin{equation}
\label{eqn:hatL}
\hat{L}:= (s_{1}\oplus d\sigma_{1}(T_{2}))\cap  (\hat{s}_{1}\oplus d\hat{\sigma}_{1}(T_{2})) =  (s_{1}\oplus d\sigma_{1}(T_{2}))\cap \hat{f}^{\perp}
\end{equation}
has rank $1$. Since $s_{1}\oplus d\sigma_{1}(T_{2})$ is constant along $T_{1}$ and sections of $\hat{s}_{1}\oplus d\hat{\sigma}_{1}(T_{2})$ differentiate into $\hat{f}^{\perp}$ along $T_{1}$, we have that $\hat{L}$ is constant along $T_{1}$. From Lemma~\ref{lem:Ldist} we deduce that the osculating complex $\hat{L}_{1}=\hat{L}$ and Proposition~\ref{prop:sphLi} implies that $\hat{f}$ has spherical curvature lines along the leaves of $T_{1}$. We have thus recovered the following result:

\begin{proposition}[\cite{B1929}]
\label{prop:ribchan}
Ribaucour transforms of channel surfaces have one family of spherical curvature lines. 
\end{proposition}

Dupin cyclides are those surfaces that are doubly channel, i.e., both families of curvature lines are circular. From Proposition~\ref{prop:ribchan}, we immediately recover the result of~\cite{B1929} that Ribaucour transforms of Dupin cyclides have two families of spherical curvature lines. However, we can in fact say more. By~\eqref{eqn:hatL}, we have that $\hat{L}_{1}\le S_{1}$ and $\hat{L}_{2}\le S_{2}$, where $S_{1}$ and $S_{2}$ are the Lie cyclides of $f$. Since $f$ is a Dupin cyclide, the Lie cylides are the Dupin cylide itself. Hence, $S_{1}$ and $S_{2}$ are constant $(2,1)$ bundles. Thus the osculating bundles $\hat{\mathcal{H}}_{1}$ and $\hat{\mathcal{H}}_{2}$ of $\hat{f}$ are both constant $(2,1)$ bundles. It thus follows from Subsection~\ref{subsec:sphsymbreak} that $\hat{f}$ projects to a Joachimsthal surface in an appropriate Euclidean space form. We thus arrive at the following proposition: 

\begin{proposition}
\label{prop:ribdupin}
Ribaucour transforms of Dupin cyclides project to Joachimsthal surfaces in appropriate Euclidean space forms. 
\end{proposition}

We now seek a converse to Proposition~\ref{prop:ribchan}. Suppose that $f:\Sigma\to \mathcal{Z}$ is a regular umbilic-free Legendre map, and without loss of generality, assume that the leaves of $T_{1}$ are spherical curvature lines. By Proposition~\ref{prop:sphLi} we have that $L:= L_{1}$ is constant along $T_{1}$. By Lemma~\ref{lem:DLflat} we have that $\mathcal{D}^{L}$ is flat. Since $L$ is parallel for $\mathcal{D}^{L}$, there locally exists a $3$-parameter family of parallel null rank $1$ subbundles $\hat{s}$ of $\mathcal{D}^{L}$ with $\hat{s}\perp L$ and $\hat{s}\cap f = \{0\}$. Define $s_{0}:= f\cap \hat{s}^{\perp}$ and $\hat{f}:= s_{0}\oplus \hat{s}$. Since $\hat{s}$ is parallel for $\mathcal{D}^{L}$, it follows that $\hat{f}^{(1)}\perp \hat{f}$. Then by Lemma~\ref{lem:regrib}, we have that $\hat{f}$ is a Legendre map. Now fix a sphere congruence $s\le f$ with $s\cap s_{0}=\{0\}$. Then it follows from the fact that $\hat{s}$ is parallel for $\mathcal{D}^{L}$ that the induced connection on $s\oplus \hat{s}$ if flat. By Lemma~\ref{lem:shatsflat} we then have that $\hat{f}$ is a Ribaucour transform of $f$. Moreover, since $\mathcal{N}^{L}(T_{1}) = 0$, we have that $\hat{s}$ is constant along $T_{1}$. Thus $\hat{f}$ is a channel surface with circular curvature lines along the leaves of $T_{1}$. We thus arrive at the following theorem: 

\begin{theorem}
\label{thm:3paramchannel}
Any regular umbilic-free Legendre map with spherical curvature lines is a Ribaucour transform of a $3$-parameter family of channel surfaces. 
\end{theorem}

\begin{remark}
Theorem~\ref{thm:3paramchannel} is a Lie geometric version of the projective geometric result found in~\cite{F2000i} that states that a surface with one family of asymptotic lines in linear complexes admits a $3$-parameter family of $W$-transformations to ruled surfaces. 
\end{remark}

In Subsection~\ref{subsec:twosphcurv} we shall prove a converse to Proposition~\ref{prop:ribdupin} using Darboux transforms. 

\subsection{Ribaucour transforms via the spherical evolution map}

In this subsection we shall see how Ribaucour pairs of surfaces with spherical curvature lines can be constructed using spherical evolution of Ribaucour pairs of Legendre curves. This leads to a geometric proof of Theorem~\ref{thm:3paramchannel}. 

Let $L:\tilde{I}\to \mathcal{S}$ be an immersion and let $A$ be the spherical evolution map based at $v_{0}$. Suppose that $C, \hat{C}:I\to \mathcal{Z}$ are two Legendre curves in $L_{0}^{\perp}$, with $L_{0}:= L(v_{0})$, such that $Ad_{A}\mathcal{N}^{L}|_{C}$ and $Ad_{A}\mathcal{N}^{L}|_{\hat{C}}$ are nowhere zero. Suppose further that $C\cap \hat{C} = c_{0}$ is a rank $1$ subbundle of $C$ and $\hat{C}$. Geometrically this means that $C$ and $\hat{C}$ envelop a common circle congruence $c_{0}$. Thus $C$ and $\hat{C}$ are a \textit{Ribaucour pair of curves} (see~\cite{BHMR2016}). From Theorem~\ref{thm:A} we have that 
\[ f(u,v)= A^{-1}(v)C(u) \quad \text{and}\quad \hat{f}(u,v) = A^{-1}(v) \hat{C}(u)\]
are Legendre maps parametrised by curvature line coordinates $(u,v)$ such that the $u$-parameter lines are spherical curvature lines. Since $f, \hat{f}\perp L$, we have that the $u$-parameter lines of $f$ and $\hat{f}$ lie on the same $1$-parameter family of spheres. Now $s_{0}:= A^{-1}c_{0}$ is a rank $1$ subbundle of $f$ and $\hat{f}$ and thus $f$ and $\hat{f}$ envelop a common sphere congruence $s_{0}$. Moreover, since $(u,v)$ are curvature line coordinates of $f$ and $\hat{f}$, it follows that $f$ and $\hat{f}$ are a Ribaucour pair of surfaces. 

\begin{proposition}
\label{prop:sphevorib}
Spherical evolution of a Ribaucour pair of Legendre curves yields a Ribaucour pair of Legendre maps. 
\end{proposition}

\begin{remark}
Following from Subsection~\ref{subsec:sphsymbreak}, we see that Proposition~\ref{prop:sphevorib} yields a construction method for Ribaucour pairs of, for example, surfaces with planar curvature lines or Monge surfaces. 
\end{remark}

By choosing a circular Legendre curve $\hat{C}:I\to \mathcal{Z}$, one has that the $u$-curvature lines of $\hat{f}(u,v) = A^{-1}(v)\hat{C}(u)$ are circular and thus $\hat{f}$ is a channel surface. In fact, such a circular $\hat{C}$ contains a constant circle $\hat{c}\le \hat{C}$, and this evolves to yield the sphere curve $\hat{s} := A^{-1}\hat{c} \le \hat{f}$ enveloped by $\hat{f}$.

Given a Legendre curve $C:I\to \mathcal{Z}$ in $L_{0}^{\perp}$, there generically exists a $3$-parameter family of circular Ribaucour transforms $\hat{C}:I\to \mathcal{Z}$ of $C$ in $L_{0}^{\perp}$. Any such transform can be constructed by choosing a constant $\hat{c}\in E_{L_{0}}$ such that $\hat{c}\cap C =\{0\}$ and defining $\hat{C}= \hat{c}\oplus c_{0}$ where $c_{0} = C\cap \hat{c}^{\perp}$. By evolving such $\hat{C}$ we recover the three parameter family of channel surfaces of Theorem~\ref{thm:3paramchannel} that are Ribaucour transforms of $f=A^{-1}C$ (see Figures \ref{fig:dupin}, \ref{fig:rib2} and \ref{fig:rib3}).

\begin{figure}
	\centering
	\begin{minipage}{0.48\textwidth}
		\centering
		\includegraphics[width=0.8\linewidth]{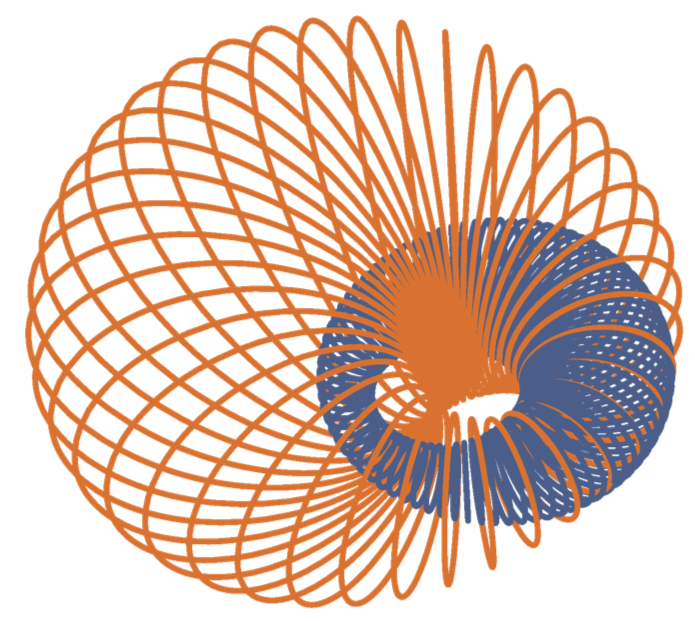}
	\end{minipage}
	\begin{minipage}{0.48\textwidth}
		\centering
		\includegraphics[width=0.8\linewidth]{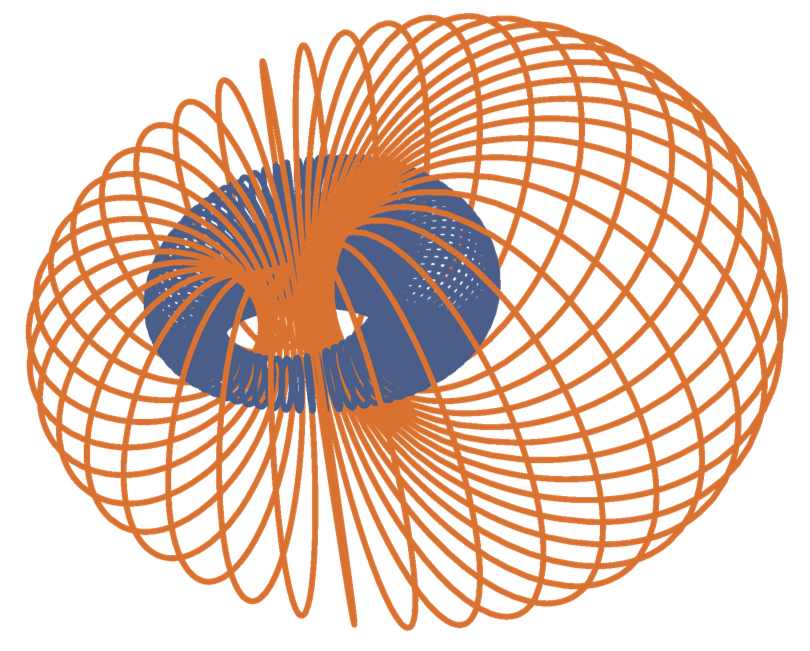}
	\end{minipage}

	\caption{Ribaucour pairs of Dupin cyclides. For Dupin cyclides, one can find an evolution map that consists of point sphere preserving Lie inversions; hence, the evolution map indeed transports the (point spheres of the) circles of the Dupin cyclide.}
	\label{fig:dupin}
\end{figure}

\begin{figure}
	\centering
	\includegraphics[width=0.9\linewidth]{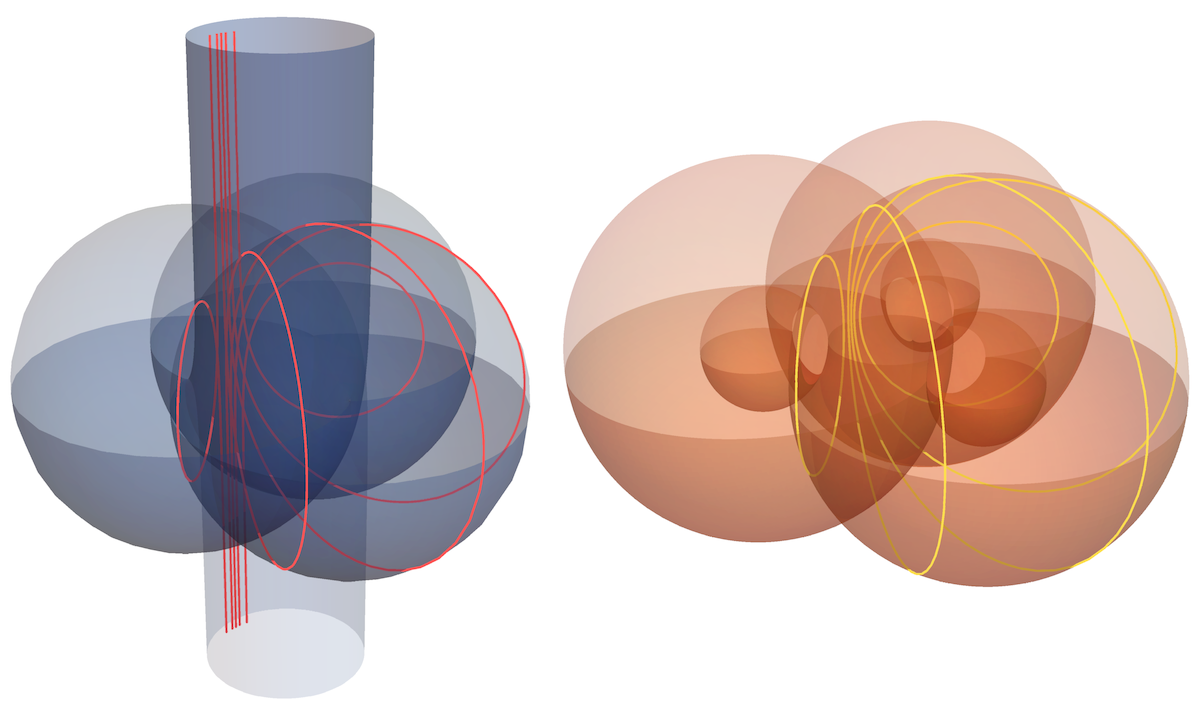}
	\caption{Compact channel surface (on the right) that is a Ribaucour transform of the cmc surface known as a bubbleton (on the left), with a few of the corresponding curvature lines highlighted. The channel surface is obtained using the spherical evolution map induced by the planar curvature lines of the cmc surface.}
	\label{fig:rib2}
\end{figure}

\begin{figure}
	\centering
	\includegraphics[width=0.9\linewidth]{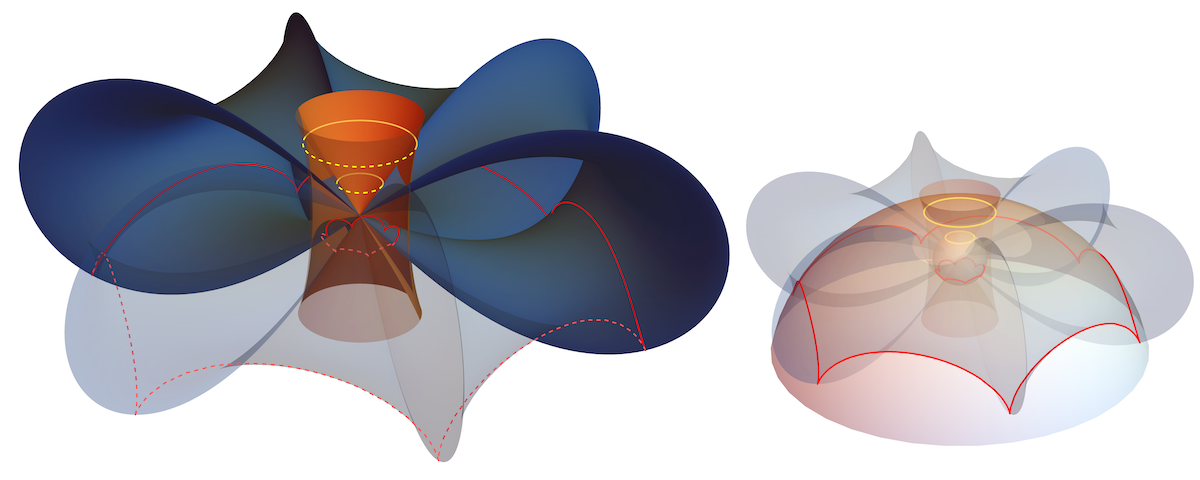}
	\caption{Surface of revolution that is a Ribaucour transform of the constant negative Gaussian curvature torus of Figure \ref{fig:first}, with two of the corresponding curvature lines highlighted. The surface of revolution is obtained using the spherical evolution map induced by the spherical curvature lines of the $K$-surface (see right figure).}
	\label{fig:rib3} 
\end{figure}

\section{Lie applicable surfaces with spherical curvature lines}
\label{sec:lieappsph}

Lie applicable surfaces are the deformable surfaces of Lie sphere geometry~\cite{MN2006}. Examples include isothermic surfaces and Guichard surfaces of conformal geometry, $L$-isothermic surfaces of Laguerre geometry and linear Weingarten surfaces of space form geometries (see~\cite{BHR2012,BHPR2019}). 

From Definition~\ref{def:lieapp}, we have that $f:\Sigma\to\mathcal{Z}$ is a Lie applicable surface if and only if there exists $\eta\in \Omega^{1}(f\wedge f^{\perp})$ satisfying $d\eta = [\eta\wedge \eta]=0$, whose associated quadratic differential $\quaddiff^{\eta}$ is non-zero. In~\cite[Proposition 3.4]{P2020} it is shown that, when $f$ is umbilic-free, the closure of $\eta$ implies $[\eta\wedge \eta]=0$ and $\eta(T_{i})\le f\wedge f_{i}$. We then have that the quadratic differential $\quaddiff^{\eta}$ may be written as 
\[ \quaddiff^{\eta} = \quaddiff^{\eta}_{1}+ \quaddiff^{\eta}_{2}\]
with $\quaddiff^{\eta}_{1}\in \Gamma (T^{*}_{1})^{2}$ and $\quaddiff^{\eta}_{2}\in \Gamma (T^{*}_{2})^{2}$.

Lie applicable surfaces have the property that they envelop isothermic sphere congruences $s\le f$. One characterisation of such sphere congruences is via the existence of a Moutard lift, that is, a section $\sigma\in \Gamma s$ such that $\sigma_{uv}\in \Gamma s$ where $(u,v)$ are curvature line coordinates of $f$ with $\frac{\partial}{\partial u}\in \Gamma T_{1}$ and $\frac{\partial}{\partial v}\in \Gamma T_{2}$. The class of Lie applicable surfaces then splits into two classes: 
\begin{itemize}
\item if $\quaddiff^{\eta}_{1}$ and $\quaddiff^{\eta}_{2}$ are both non-zero then $f$ envelops a pair of isothermic sphere congruences that separate the curvature spheres harmonically. Such surfaces are called $\Omega$-surfaces. 
\item if $\quaddiff^{\eta}_{i}$ vanishes for some $i\in\{1,2\}$, then the curvature sphere congruence $s_{i}$ is isothermic. Such surfaces are called $\Omega_{0}$-surfaces. 
\end{itemize}

In this section we shall consider the case of an umbilic-free Legendre map having at least one family of spherical curvature lines, whilst additionally being a Lie applicable surface. In~\cite{B2001}, isothermic surfaces with one or two families of spherical curvature lines are characterised in terms of differential equations on the Calapso potential. In~\cite{MN1999i} it is shown that surfaces with two families of planar curvature lines are $L$-isothermic and in~\cite{F2002} it is shown that surfaces with two families of spherical curvature lines are multiply Lie applicable. In~\cite{PS2018} it shown that channel surfaces are examples of $\Omega_{0}$-surfaces and in~\cite{HPPwip} it is shown that channel surfaces that are additionally $\Omega$-surfaces are rotational surfaces. 

\subsection{Lie applicable surfaces with one family of spherical curvature lines}

Suppose that $f:I_{1}\times I_{2}\to \mathcal{Z}$ is a regular umbilic-free Legendre map parametrised by curvature line coordinates $(u,v)$ and, without loss of generality, assume that the $u$-parameter lines are spherical. By Proposition~\ref{prop:sphLi} this implies that the osculating sphere complex $L:= L_{1}$ is constant along the $u$-parameter lines, and thus we may consider it as a map $I_{2}\to S$. Using Theorem~\ref{thm:A} we may write $f$ as 
\[ f = A^{-1}C:I_{1}\times I_{2}\to \mathcal{Z}\]
where $A$ is the spherical evolution map along $L$ based at $v_{0}\in I_{2}$ and $C:I_{1}\to \mathcal{Z}$ is a Legendre curve in $L_{0}^{\perp}$ with $L_{0}:= L(v_{0})$. 

Consider $\eta \in \Omega^{1}(f\wedge f^{\perp})$. We may write 
\[ \eta = \eta_{f\wedge L} + \eta_{f\wedge L^{\perp}}\]
where $\eta_{f\wedge L}\in \Omega^{1}(f\wedge L)$ and $\eta_{f\wedge L^{\perp}} \in \Omega^{1}(f\wedge (f^{\perp}\cap L^{\perp}))$. By Lemma~\ref{lem:osc} we have that $f_{1} = f\oplus L_{2}$ and $f_{2} = f\oplus L_{1}$. Thus the conditions $\eta(T_{i})\le f\wedge f_{i}$ imply that $\eta_{f\wedge L}(T_{1})=0$ and $\eta_{f\wedge L^{\perp}}(T_{2}) \le \wedge^{2} f$. One then deduces that
\begin{align*}
\quaddiff^{\eta}_{1}(X,Y) &= \tr(f\to f: \sigma \mapsto \eta_{f\wedge L^{\perp}}(X)d_{Y}\sigma),\\
\quaddiff^{\eta}_{2}(X,Y) &= \tr(f\to f: \sigma \mapsto \eta_{f\wedge L}(X)d_{Y}\sigma). 
\end{align*}

The following result shows that if a regular Legendre map $f$ has spherical curvature lines along the leaves of $T_{i}$, then the curvature sphere congruence $s_{i}$ cannot be isothermic: 

\begin{lemma}
\label{lem:Q1neq0}
Suppose that $\eta$ is closed with $\quaddiff^{\eta}\neq 0$. Then $\quaddiff^{\eta}_{1}\neq 0$.
\end{lemma}
\begin{proof}
If $\quaddiff^{\eta}_{1} =0$, then $s_{1}$ is isothermic and there exists a Moutard lift $\sigma_{1}\in \Gamma s_{1}$, i.e., $\sigma_{1,uv}\in \Gamma s_{1}$. Now $(\sigma_{1,v})_{u}, (\sigma_{1,v})_{v}\in \Gamma S_{1}$ and one deduces that $L_{1} = \langle \sigma_{1,v}\rangle$. Since $L_{1}$ is constant along $T_{1}$, it follows that $\sigma_{1,uv} = (\sigma_{1,v})_{u}=0$. Fixing some non-zero section $\sigma_{2}\in \Gamma s_{2}$, we have that  $\sigma_{1,u}  = \alpha \sigma_{1} + \beta \sigma_{2}$ for some functions $\alpha$, $\beta$. The condition that $\sigma_{1,uv}=0$ implies that $\alpha = 0$ and either $\beta = 0$ or $\sigma_{2,v}=0$. This implies that one of the curvature spheres of $f$ does not immerse, contradicting our regularity assumption on $f$. 
\end{proof}

Using the spherical evolution map $A$ we may write
\[ \eta = Ad_{A^{-1}}\cdot (\zeta_{C\wedge L_{0}} + \zeta_{C\wedge L_{0}^{\perp}}),\]
where 
\[ \zeta_{C\wedge L_{0}} := Ad_{A}\cdot \eta_{f\wedge L}\in \Omega^{1}(C\wedge L_{0})\: \text{ and }  \:\zeta_{C\wedge L_{0}^{\perp}} := Ad_{A}\cdot \eta_{f\wedge L^{\perp}}\in \Omega^{1}(C\wedge (C^{\perp}\cap L_{0}^{\perp})).\] 
Since $\eta_{f\wedge L}(T_{1})=0$ and $\eta_{f\wedge L^{\perp}}(T_{2})\le\wedge^{2}f$, it follows that $\zeta_{C\wedge L_{0}}(T_{1})=0$, and $\zeta_{C\wedge L_{0}^{\perp}}(T_{2})\le \wedge^{2}C$. Moreover, we have that 
\begin{align*}
\mathcal{Q}^{\eta}_{1}(X,Y) &= \tr(C\to C: \theta \mapsto \zeta_{C\wedge L_{0}^{\perp}}(X)d_{Y}\theta), \\
 \mathcal{Q}^{\eta}_{2}(X,Y) &= \tr(C\to C: \theta \mapsto  \zeta_{C\wedge L_{0}}(X)d_{Y}\theta). 
 \end{align*}
Towards a geometric classification of Lie applicable surfaces with one family of spherical curvature lines (see Theorem~\ref{thm:lieappsph}), our strategy for the rest of this subsection will be to show that when $\eta$ is closed, we may gauge $\eta$ so that $\zeta_{C\wedge L_{0}^{\perp}}$ yields a polarisation for $C$. The other consequences of $\eta$ being closed then give rise to linear conserved quantities for this polarisation. 

Using~\eqref{eqn:A-1gauge}, the closure of $\eta$ is given by
\[d\eta = Ad_{A^{-1}}\cdot ((d+Ad_{A}\mathcal{N}^{L})(\zeta_{C\wedge L_{0}} + \zeta_{C\wedge L_{0}^{\perp}})).\]
Thus $\eta$ is closed if and only if 
\[ (d+Ad_{A}\mathcal{N}^{L})(\zeta_{C\wedge L_{0}} + \zeta_{C\wedge L_{0}^{\perp}}) = 0.\]
By splitting this into $C\wedge L_{0}$ and $C\wedge L_{0}^{\perp}$ parts, we see that this is equivalent to 
\begin{align}
d\zeta_{C\wedge L_{0}} + [Ad_{A}\mathcal{N}^{L}\wedge  \zeta_{C\wedge L_{0}^{\perp}}] &=0 \label{eqn:dcl0}\\
d \zeta_{C\wedge L_{0}^{\perp}} + [Ad_{A}\mathcal{N}^{L}\wedge  \zeta_{C\wedge L_{0}}] &=0.\label{eqn:dcl0perp}
\end{align}
Since $Ad_{A}\mathcal{N}^{L}(T_{1})=0$ and $\zeta_{C\wedge L_{0}}(T_{1})=0$, \eqref{eqn:dcl0perp} becomes
\begin{equation} 
\label{eqn:dcl0perp2}
d \zeta_{C\wedge L_{0}^{\perp}} = 0.
\end{equation}

Consider a gauge transformation 
\[ \tilde{\eta} = \eta - d\tau = Ad_{A^{-1}}\cdot (\tilde{\zeta}_{C\wedge L_{0}} + \tilde{\zeta}_{C\wedge L_{0}^{\perp}})\]
of $\eta$. We may write $\tau = A^{-1}\cdot \rho$ for some $\rho \in \Gamma (\wedge^{2} C)$. It is then straightforward to see that 
\[ \tilde{\zeta}_{C\wedge L_{0}} = \zeta_{C\wedge L_{0}} - [Ad_{A^{-1}}\mathcal{N}^{L}, \rho] \quad \text{and}\quad 
\tilde{\zeta}_{C\wedge L^{\perp}_{0}} =\zeta_{C\wedge L^{\perp}_{0}} - d\rho.\]
Now, $\zeta_{C\wedge L_{0}^{\perp}}(T_{2})\le \wedge^{2}C$, and since $C$ doesn't depend on $I_{2}$, it follows that $d\rho(T_{2})\le \wedge^{2} C$. Thus, we may choose $\rho$ so that $\tilde{\zeta}_{C\wedge L_{0}^{\perp}}(T_{2})=0$. From~\eqref{eqn:dcl0perp2} and Lemma~\ref{lem:Q1neq0}, we can then deduce\footnote{$\tilde{\xi}(T_{2})=0$ implies that $\tilde{\xi} = g \, du$ for some smooth function $g$. Equation~\eqref{eqn:dcl0perp2} then implies that $g_{v} = 0$, and thus $g$ is a function of $u$ only.}  that $[\tilde{\xi}]$ with $\tilde{\xi}:=\tilde{\zeta}_{C\wedge L_{0}^{\perp}}$ is a polarisation of $C$ with quadratic differential $\mathcal{Q}^{\tilde{\xi}} = \mathcal{Q}^{\eta}_{1}$. 

On the other hand, \eqref{eqn:dcl0} can be reformulated as
\[(d+ t\zeta_{C\wedge L_{0}^{\perp}})(Ad_{A}\mathcal{N}^{L} + t\zeta_{C\wedge L_{0}}) = 0.\]
Therefore we arrive at the following result: 

\begin{proposition}
\label{prop:etaxiclosed}
$\eta = Ad_{A^{-1}}\cdot (\zeta_{C\wedge L_{0}} + \zeta_{C\wedge L_{0}^{\perp}})$ is closed with non-zero quadratic differential $\quaddiff^{\eta}$ if and only if, after an appropriate gauge transformation, $[\xi]$ with $\xi:=  \zeta_{C\wedge L_{0}^{\perp}}$ is a polarisation of $C$ and 
\begin{equation}
\label{eqn:dtxi}
(d+ t\xi)(Ad_{A}\mathcal{N}^{L} + t\zeta_{C\wedge L_{0}}) = 0.
\end{equation}
\end{proposition}

Fix a unit length vector $l_{0}\in L_{0}$. Then 
\[ (Ad_{A}\mathcal{N}^{L} + t\zeta_{C\wedge L_{0}})(\tfrac{\partial}{\partial v}) = l_{0}\wedge \nu(t)\]
for some linear $\nu(t)\in (\Gamma L_{0}^{\perp})[t]$. Then~\eqref{eqn:dtxi} implies that 
\begin{equation}
\label{eqn:pcqsec}
(d+ t\xi)|_{T_{1}}\nu(t) = 0.
\end{equation}
Hence, $\nu\in \Gamma W$, where $W$ is the space of linear conserved quantities of $(d+ t\xi)|_{T_{1}}$. By differentiating~\eqref{eqn:pcqsec} repeatedly with respect to $v$, we find that 
\[\nu, \nu_{v}, \nu_{vv}, ...\in \Gamma W.\] 
We thus have that 
\[ \nu(0), (\nu(0))_{v}, (\nu(0))_{vv},... \in \Gamma W(0),\]
with $\nu(0) = Ad_{A}\mathcal{N}^{L}(\tfrac{\partial}{\partial v})l_{0}$. Now 
\begin{align*} 
L &= A^{-1}L_{0}, \\
\langle l, l_{v}\rangle  &= A^{-1}\langle l_{0},  \nu(0)\rangle,\\ 
\mathcal{H}_{1} =  \langle l, l_{v}, l_{vv}\rangle &=  A^{-1}\langle l_{0},    \nu(0), (\nu(0))_{v}\rangle,
\end{align*}
and so on. In particular, since $\mathcal{H}_{1}$ has rank $3$, we must have that  $\dim W(0)\ge 2$. Hence $\dim W \ge 2$. On the other hand, by Proposition~\ref{prop:4lcq}, since $C$ is non-circular, we have that $\dim W \le 3$. Therefore, we arrive at two possibilities:
\begin{itemize}
\item $\dim W =2$: then $\mathcal{H}_{1}$ is constant, implying that $f$ has two families of spherical curvature lines by Proposition~\ref{prop:Hiconst}. 
\item $\dim W =3$: then $L$ lives in the constant $4$-dimensional space $L_{0}\oplus W(0)$ and $C$ is constrained elastic curve with respect to $W(0)$ (see Definition~\ref{def:conelas}). 
\end{itemize}

Conversely, suppose that $L$ lives in a constant subspace $L_{0}\oplus_{\perp} W_{0}$, with $W_{0}$ either $2$- or $3$-dimensional. We assume that there exists a polarisation $[\xi]$ of $C$ such that $d+t\xi$ admits a space $W$ of linear conserved quantities with $W_{0}\le W(0)$. Depending on the dimension of $W_{0}$ we have two cases: 
\begin{itemize}
\item If $W_{0}$ is $2$-dimensional, then by Lemma~\ref{lem:2dimlcq} there always exists such a polarisation $[\xi]$. 
\item If $W_{0}$ is $3$-dimensional, then $C$ is a constrained elastic curve with respect to $W_{0}$. 
\end{itemize}

Now $Ad_{A}\mathcal{N}^{L} \in \Gamma T_{2}^{*}\otimes (L_{0}\wedge W_{0})$ and we can find a $1$-form $\zeta_{C\wedge L_{0}}\in \Gamma T_{2}^{*}\otimes(L_{0}\wedge C) $ such that 
\begin{equation}
\label{eqn:dtxi2}
(d+t\xi)(Ad_{A}\mathcal{N}^{L} + t \zeta_{C\wedge L_{0}}).
\end{equation}
Define
\[ \eta = Ad_{A^{-1}}\cdot (\zeta_{C\wedge L_{0}} + \xi) \in \Omega^{1}(f\wedge f^{\perp}).\]
Since $\xi$ is a polarisation of $C$ and~\eqref{eqn:dtxi2} holds, we have by Proposition~\ref{prop:etaxiclosed} that $\eta$ is closed and $\mathcal{Q}^{\eta}\neq 0$. Hence, $f$ is a Lie applicable surface. 
We have thus arrived at the following theorem: 

\begin{theorem}
\label{thm:lieappsph}
Suppose that $f=A^{-1}C:I_{1}\times I_{2}\to\mathcal{Z}$ is a regular umbilic-free Legendre map with one family of spherical curvature lines, where $A$ is the spherical evolution map along $L$ based at $v_{0}\in I_{2}$ and $C$ is a Legendre curve in $L_{0}^{\perp}$. Then $f$ is Lie applicable if and only if either the other family of curvature lines are also spherical or 
$L$ lies in a constant $4$-dimensional space $L_{0}\oplus_{\perp} W_{0}$ and $C$ is a constrained elastic curve with respect to $W_{0}$.
\end{theorem}

\begin{figure}
	\centering
	\includegraphics[width=0.98\textwidth]{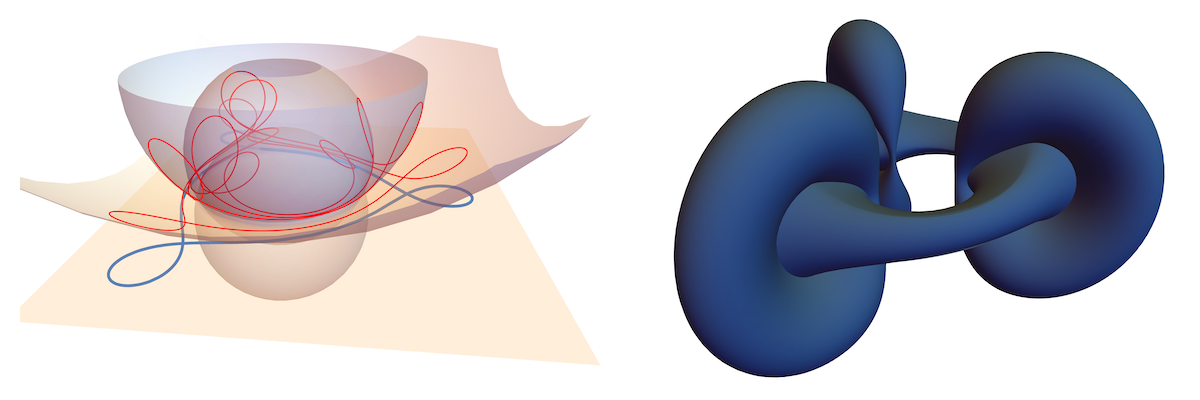}
	\caption{An example of a Lie applicable torus with spherical curvature lines (on the right) lying on spheres whose centres all belong to the same line. The initial curve is in bicycle correspondence with a circle (on the left).}
	\label{fig:lieApp1}
\end{figure}

\begin{figure}
	\centering
	\includegraphics[width=0.98\textwidth]{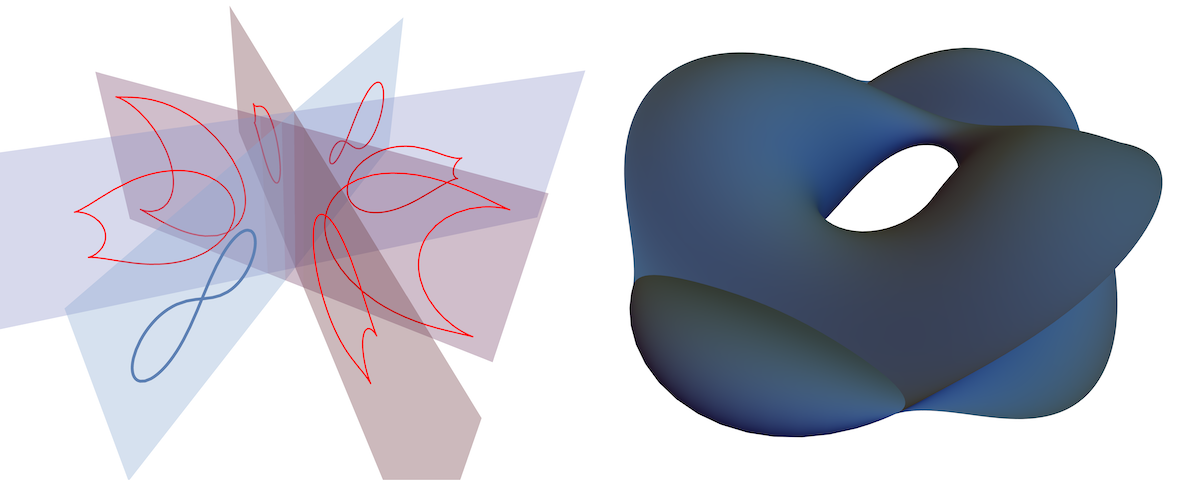}
	\caption{An example of a Lie applicable torus with planar curvature lines (on the right) lying on planes that are orthogonal to a fixed plane. The initial curve is an elastic lemniscate (on the left).}
	\label{fig:lieApp2}
\end{figure}

The condition that $L$ lies in a constant $4$-dimensional space $L_{0}\oplus_{\perp} W_{0}$ has different interpretations depending on the signature of $W_{0}$. 

If $W_{0}$ contains a unit timelike vector $\mathfrak{p}$, then $S= \pi_{\mathfrak{p}}(L)$ lies in a constant $3$-dimensional subspace of $\langle\mathfrak{p}\rangle^{\perp}$. Using~\cite[Section 1.8.5]{H2003}, one then determines the nature of $S$ after projecting to an appropriate Euclidean geometry:  
\begin{itemize} 
\item if $W_{0}$ has signature $(1,2)$ then $S$ is a family of spheres whose centres lie on a fixed line (see Figure \ref{fig:lieApp1});
\item if $W_{0}$ has signature $(2,1)$ then $S$ is a family of planes intersecting a fixed point;
\item if $W_{0}$ has signature $(1,1,1)$ then $S$ is a family of planes that intersect a fixed plane orthogonally (see Figure \ref{fig:lieApp2}). 
\end{itemize}

If $(L_{0}\oplus W_{0})^{\perp}$ contains a unit timelike vector $\mathfrak{p}$ then Proposition~\ref{prop:planorth} tells us that $S=\pi_{\mathfrak{p}}(L)$ intersects the point sphere map $f\cap\langle \mathfrak{p}\rangle^{\perp}$ orthogonally. After projecting to an appropriate Euclidean geometry we have the following additional cases: 
\begin{itemize}
\item If $W_{0}$ has signature $(3,0)$, then $S$ belongs to a hyperbolic sphere complex. 
\item If $W_{0}$ has signature $(2,0,1)$, then we may choose a space form vector $\mathfrak{q}\in (L_{0}\oplus W_{0})^{\perp}$. It then follows from Proposition~\ref{prop:monge} that $f$ projects to a Monge surface in $\mathfrak{Q}^{3}$. 
\end{itemize}

The only other possibility that remains is that $W_{0}$ has signature $(1,0,2)$. One then has that the sphere complexes defined by $L$ contain a fixed contact element, namely, the null $2$-dimensional subspace contained in $W_{0}$. 

A straightforward implication of Theorem~\ref{thm:lieappsph} is: 

\begin{theorem}\label{thm:cool}
If a Lie applicable surface has exactly one family of spherical curvature lines then all members of this family are Lie sphere transforms of a certain constrained elastic curve. 
\end{theorem}

\subsection{Surfaces with two families of spherical curvature lines}
\label{subsec:twosphcurv}
As we saw in Theorem~\ref{thm:lieappsph}, surfaces with two families of spherical curvature lines are Lie applicable. In this subsection we shall give an alternative proof of this using a curved flat characterisation developed in~\cite{BP2020}. 

Suppose that $f$ is a regular umbilic-free Legendre map and suppose that both families of curvature lines of $f$ are spherical. Then by Proposition~\ref{prop:sphLi} we have that both osculating complexes $L_{i}$ are constant along $T_{i}$. Now we can define a bundle of $(2,2)$-planes $W:=  (L_{1}\oplus L_{2})^{\perp}$ with $f\le W$. This yields a splitting of the trivial bundle $\underline{\mathbb{R}}^{4,2} = W\oplus W^{\perp}$ and a splitting of the trivial connection 
\[ d = \mathcal{D}^{W}+ \mathcal{N}^{W},\]
where $\mathcal{D}^{W}$ is the sum of the induced connections on $W$ and $W^{\perp}$ and $\mathcal{N}^{W}\in \Omega^{1}(W\wedge W^{\perp})$. Now fix unit length sections $l_{i}\in \Gamma L_{i}$. Then 
\[ \mathcal{N}^{W} = l_{1}\wedge dl_{1} + l_{2}\wedge dl_{2},\]
and since $d|_{T_{i}}l_{i} = 0$, we have that $d\mathcal{N}^{W}=0$. Hence, $W$ is a curved flat in $G_{2,2}(\mathbb{R}^{4,2})$, where $G_{2,2}(\mathbb{R}^{4,2})$ denotes the Grassmannian of $(2,2)$ planes in $\mathbb{R}^{4,2}$. It then follows from~\cite[Corollary 3.8]{BP2020} that $f$ is Lie applicable. To summarise:

\begin{proposition}
Suppose that both families of curvature lines of a regular umbilic-free Legendre map $f$ are spherical. Then $f$ is Lie applicable and $(L_{1}\oplus L_{2})^{\perp}$ is a curved flat in $G_{2,2}(\mathbb{R}^{4,2})$. 
\end{proposition}

From~\cite[Theorem 3.9]{BP2020} we have that the $\mathcal{D}^{W}$-parallel null rank 2 subbundles of $W$ form two $1$-parameter families of Lie applicable Legendre maps $\{f_{\alpha}\}$,$\{\hat{f}_{\beta}\}$ parametrised by $\alpha, \beta \in \mathbb{R}P^{1}$, such that for some $m\in\mathbb{R}^{\times}$ any pair $(f_{\alpha},\hat{f}_{\beta})$ is an $m$-Darboux pair. Without loss of generality, assume that $f\in \{f_{\alpha}\}$.  Note that members of the same family are complementary, i.e., 
\[ f_{\alpha_{1}}\cap f_{\alpha_{2}} = \{0\} = \hat{f}_{\beta_{1}}\cap \hat{f}_{\beta_{2}} ,\]
for all $\alpha_{1},\alpha_{2},\beta_{1},\beta_{2}\in \mathbb{R}P^{1}$. From Theorem~\ref{thm:Lsph}, we can deduce that all members of these families have two families of spherical curvature lines.

Suppose that $f$ projects to a Joachimsthal surface in an appropriate Euclidean geometry. Then from Subsection~\ref{subsec:sphsymbreak}, we have that the osculating bundles $\mathcal{H}_{1}$ and $\mathcal{H}_{2}$ are constant $(2,1)$ spaces. Thus $V_{i}:=\mathcal{H}_{i}\cap W$ are $(1,1)$ subbundles of $W$ and we denote by $\hat{s}^{j}_{i}\le V_{i}$ the corresponding null line bundles, for $j\in \{1,2\}$. Since $\mathcal{H}_{i}$ is constant it follows that each $\hat{s}_{i}^{j}$ is $\mathcal{D}^{W}$-parallel and, since $L_{i}$ is constant along $T_{i}$, it follows that $\hat{s}_{i}^{j}$ is constant along $T_{i}$. Note that we cannot have that any of $\hat{s}_{i}^{j}$ are constant since this would contradict that $\mathcal{H}_{i}$ are rank $3$. By taking various combinations $\hat{f}^{j_{1}j_{2}}=\hat{s}_{1}^{j_{1}}\oplus  \hat{s}_{2}^{j_{2}}$ with $j_{1},j_{2}\in \{1,2\}$, we obtain four Legendre maps contained in $W$ that are $\mathcal{D}^{W}$-parallel, and since $\hat{s}_{i}^{j_{1}}$ is constant along $T_{i}$, it follows that $\hat{f}^{j_{1}j_{2}}$ are Dupin cyclides. We also have that the pair $\hat{f}^{11}$, $\hat{f}^{22}$ are complementary and both intersect the complementary pair $\hat{f}^{12}$, $\hat{f}^{21}$. Without loss of generality, assume that $\hat{f}^{11},\hat{f}^{22} \in \{f_{\alpha}\}$ and 
$\hat{f}^{12},\hat{f}^{21}\in \{\hat{f}_{\beta}\}$. Thus, $\hat{f}^{12}$ and $\hat{f}^{21}$ are Darboux transforms of $f$. We thus arrive at the following theorem that gives a converse to Proposition~\ref{prop:ribdupin}: 

\begin{theorem}
\label{thm:joachimsthal}
Joachimsthal surfaces are Darboux transforms of two Dupin cyclides. 
\end{theorem}

\begin{remark}\label{rem:joachimsthal}
Joachimsthal surfaces have one family of curvature lines that lie on spheres whose centres lie on a fixed line. 
Geometrically the two Dupin cyclides in Theorem~\ref{thm:joachimsthal} are the fixed line, parametrised as the intersection points of the spheres with this fixed line. 
\end{remark}

\begin{remark}
Remark~\ref{rem:joachimsthal} immediately allows one to deduce that cmc tori with planar curvature lines, including the Wente torus, can locally be obtained as a Darboux transform of the circular cylinder using the fact that the parallel constant Gaussian curvature tori are Joachimsthal surfaces.
	We will defer a more detailed discussion of this fact to a future work.
	For an illustration of this, see Figure~\ref{fig:wente}.
\end{remark}

\begin{figure}
	\centering
	\includegraphics[width=0.9\textwidth]{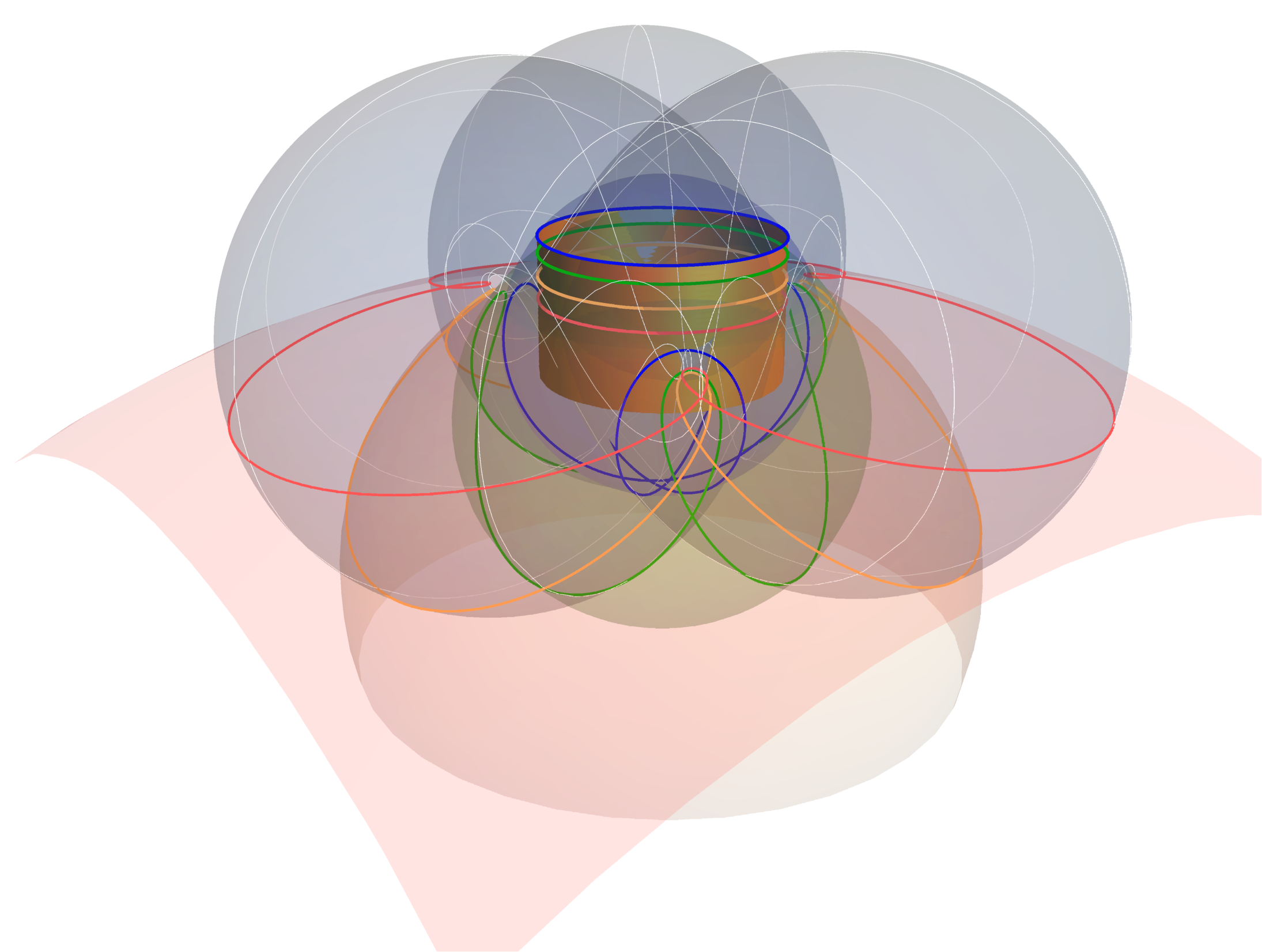}
	\caption{Wente torus obtained as a Darboux transform of the circular cylinder, with a few of the corresponding curvature lines highlighted. The circular cylinder was obtained using the spherical evolution map induced by the spherical curvature lines of the Wente torus.}
	\label{fig:wente}
\end{figure}

\subsection{Calapso transforms} 
Suppose that $f$ is a regular umbilic-free Lie applicable Legendre map with closed $1$-form $\eta$. The Lie cyclide splitting 
of the trivial bundle gives rise to a splitting $\wedge^{2}\underline{\mathbb{R}}^{4,2} = \mathfrak{h}\oplus\mathfrak{m}$, where 
\[\mathfrak{h} := S_{1}\wedge S_{1}\oplus S_{2}\wedge S_{2} \quad \text{and}\quad \mathfrak{m} := S_{1}\wedge S_{2}.\]
We may then write $\eta = \eta_{\mathfrak{h}} + \eta_{\mathfrak{m}}$ where $\eta_{\mathfrak{h}} \in \Omega^{1}(\mathfrak{h})$ and $\eta_{\mathfrak{m}}\in \Omega^{1}(\mathfrak{m})$. In~\cite{P2020} it is shown that within the gauge orbit of $[\eta]$ there is a distinguished gauge potential $\eta^{mid}$ called the middle potential with the property $\eta^{mid}_{\mathfrak{m}}\in \Omega^{1}(\wedge^{2}f)$. 
Now suppose that $f^{t}:= T(t)f$ is a Calapso transform of $f$ where $T(t)$ satisfies 
\[ T(t)\cdot (d+t\eta^{mid})=d. \]
In~\cite[Lemma 4.6]{P2020} it is shown that the Lie cyclides of $f^{t}$ are given by 
\[ S_{1}^{t} = T(t)S_{1} \quad \text{and}\quad S_{2}^{t}= T(t)S_{2}\]
and the subsequent splitting of the trivial bundle $d= \mathcal{D}^{t}+\mathcal{N}^{t}$ satisfies 
\[ \mathcal{D}^{t} = T(t)\cdot(\mathcal{D} + t\eta^{mid}_{\mathfrak{h}}) \quad \text{and}\quad \mathcal{N}^{t}= T(t)\cdot (\mathcal{N}+ t\eta^{mid}_{\mathfrak{m}}).\]
Since $\eta^{mid}_{\mathfrak{m}}\in \Omega^{1}(\wedge^{2}f)$ it follows that 
\[ \mathcal{N}^{t}(T(t)L_{i}) = T(t)(\mathcal{N}+ t\eta^{mid}_{\mathfrak{m}})L_{i} = 0.\]
Thus $L_{i}^{t}:= T(t)L_{i}$ are the osculating complexes of $f^{t}$. Moreover, since $\eta(T_{i})\le f\wedge f_{i}$ and $L_{1}\le f_{2}$ and $L_{2}\le f_{1}$, it follows that 
\[ d_{X_{i}}l_{i}^{t} = T(t)(d_{X_{i}}+t\eta^{mid}(X_{i}))l_{i} = T(t)d_{X_{i}}l_{i},\]
for sections $l_{i}\in \Gamma L_{i}$ and $X_{i}\in \Gamma T_{i}$. Applying Proposition~\ref{prop:sphLi}, we arrive at the following result: 

\begin{proposition}
A Lie applicable surface $f$ has spherical curvature lines along the leaves of $T_{i}$ if and only if any Calapso transform $f^{t}$ has spherical curvature lines along the leaves of $T_{i}$. 
\end{proposition}

\section{Final remarks}

In this paper, we have discussed how every surface with spherical curvature lines is generated by a pair of \emph{initial data}: a Legendre curve and a curve of suitable Lie sphere transformations.
Such observation yields a characterisation of Lie applicable surfaces with spherical curvature lines in terms of constrained elastic curves.
The generality of the theory developed in this paper leads to additional future questions, some of which we list below:
\begin{enumerate}
	\item The Lie sphere geometric characterisation of constrained elastic curves depends on suitable choices of point sphere complex $\mathfrak{p}$ and space form vector $\mathfrak{q}$ in $W_0$, raising the question of finding the characteristics of such curves when projected with respect to different choices of $\mathfrak{p},\mathfrak{q} \in W_0$.
	\item A closer look at the figures in this paper reveals that many Lie applicable surfaces with spherical curvature lines have admissible singularities such as cuspidal edges or swallowtails.
		This is not so surprising since many of the well-known Lie applicable surfaces such as pseudospherical surfaces and linear Weingarten surfaces are known to admit certain types of singularities.
		The theory developed here assumes that the contact lifts are immersed, and hence are fronts, allowing one to contemplate the relationship between the initial data of the Lie applicable surfaces with spherical curvature lines and recognising the types of singularities appearing on them.
	\item A recent work \cite{ogata_ribaucour_nodate} reveals that the contact lifts of Ribaucour transforms of surface of revolutions can fail to immerse, and develop frontal singularities.
		The Ribaucour transforms treated in this work should also have spherical curvature lines; therefore, one could also consider a similar problem of finding the conditions on the initial data for the resulting surface to have frontal singularities.
	\item Within the class of Lie applicable surfaces, many of the well-known surface classes can be determined via the existence of polynomial conserved quantities \cite{BHR2010, BHR2012, BHPR2019}.
		One could consider the additional conditions on the initial data so that the resulting surface admits polynomial conserved quantities, thereby obtaining a characterisation of the surfaces in the various subclasses with spherical curvature lines.
	\item Many of the examples included in this paper are compact: 
A closed initial spherical Legendre curve and a suitable periodic spherical evolution map will result in a closed compact surface with spherical curvature lines.
		Therefore, finding the conditions on the family of elliptic linear sphere complexes $L : \tilde{I} \to \mathbb{R}^{4,2}$ so that the corresponding spherical evolution map becomes periodic is key to obtaining (Lie applicable) tori with spherical curvature lines.
		\end{enumerate}


\end{document}